\documentclass[leqno,11pt, a4]{amsart}
\usepackage[usenames,dvipsnames,svgnames,table]{xcolor}
\usepackage[active]{srcltx}
\usepackage{pb-diagram}
\usepackage{dsfont}
\usepackage{mathrsfs}
\usepackage{amsmath}
\usepackage{amssymb}
\usepackage{amscd}
\usepackage{amsthm}
\usepackage[latin1]{inputenc}
\usepackage{graphics}
\usepackage{varioref}

\usepackage[latin1]{inputenc}
\usepackage{graphics}
\usepackage{pdfsync}

\labelformat{enumi}{(#1)}

\newtheorem{teo}[equation]{Theorem}
\newtheorem{defin}[equation]{Definition}

\newtheorem{prop}[equation]{Proposition}
\newtheorem{cor}[equation]{Corollary}
\newtheorem{lemma}[equation]{Lemma}

\newtheoremstyle{named}{}{}{\itshape}{}{\bfseries}{.}{.5em}{\thmnote{#3}#1}
\theoremstyle{named}

\newcommand{\ga}{\gamma}

\newcommand{\OO}{\mathcal{O}}
\newcommand{\meno}{^{-1}}

\newcommand{\ext}{\operatorname{ext}} 

\newcommand{\liu}{\mathfrak{u}}
\newcommand{\liek}{\mathfrak{k}}

\newcommand{\lieg}{\mathfrak{g}}

\newcommand{\liep}{\mathfrak{p}}
\newcommand{\lieq}{\mathfrak{q}}

\newcommand{\liez}{\mathfrak{z}}
\newcommand{\lia}{\mathfrak{a}}
\newcommand{\liem}{\mathfrak{m}}
\newcommand{\lien}{\mathfrak{n}}

\newcommand{\spam}{\,\operatorname{span}\, }

\newcommand{\alfa}{\alpha}

\newcommand{\vacuo}{\emptyset}

\newcommand{\la}{\lambda}

\newcommand{\enf}{\emph}

\newcommand{\desudt}[1] []      {\dfrac {\mathrm {d} #1 }{\mathrm {dt}}}
\newcommand{\desudtzero}        {\desudt \bigg \vert _{t=0} }

\newcommand{\restr}[1]          {\vert_{#1}}

\newcommand{\Ad}{\operatorname{Ad}}
\newcommand{\ad}{{\operatorname{ad}}}
\newcommand{\sx}{\langle}
\newcommand{\xs}{\rangle}
\newcommand{\scalo}{\sx \cdot , \cdot \xs}
\newcommand{\relint}{\operatorname{relint}}

\newcommand{\Gl}{\operatorname{GL}}

\newcommand{\End}{\operatorname{End}}

\newcommand{\cds}{\cdots}
\newcommand{\cd}{\cdot}
\renewcommand{\setminus}{-}

\newcommand{\ra}{\rightarrow}
\newcommand{\lra}{\longrightarrow}
\newcommand{\C}{\mathbb{C}}
\newcommand{\R}{\mathds{R}}

\newcommand{\om}{\omega}

\renewcommand{\phi}{\varphi}

\newcommand{\maxm}[2]{\mathrm{Max}_{#2}(#1)}

\newcommand{\Crit}{\operatorname{Crit}}

\newcommand{\roots}{\Delta}

\newcommand{\CF}{C_F}

\newcommand{\meo}{\end{document}}
\newcommand{\mup}{\mu_\liep}
\newcommand{\mua}{\mu_\lia}
\newcommand{\mupb}{\mu_\liep^\beta}

\newcommand{\simple}{\Pi} 
\begin{document}
\title{Projective representations of real reductive Lie groups and the gradient map}
\author{Leonardo Biliotti}
\address{(Leonardo Biliotti) Dipartimento di Scienze Matematiche, Fisiche e Informatiche \\
          Universit\`a di Parma (Italy)}
\email{leonardo.biliotti@unipr.it}
\begin{abstract}
Let $G$ be a connected semisimple noncompact real  Lie group and let $\rho: G \lra \mathrm{SL}(V)$ be a representation on a finite dimensional vector space $V$ over $\R$, with $\rho(G)$ closed in $\mathrm{SL}(V)$.   Identifying $G$ with $\rho(G)$, we assume there exists a $K$-invariant scalar product $\mathtt g$ such that $G=K\exp(\liep)$, where $K=\mathrm{SO}(V,\mathtt g)\cap G$, $\liep=\mathrm{Sym}_o (V,\mathtt g)\cap \lieg$ and $\lieg$ denotes the Lie algebra of $G$. Here $\mathrm{Sym}_o (V,\mathtt g)$ denotes the set  of symmetric endomorphisms with trace zero. Using the $G$-gradient map techniques we analyze the natural projective representation of $G$ on $\mathbb P(V)$.
\end{abstract}

%
\thanks{The author was partially supported by PRIN  2017
   ``Real and Complex Manifolds: Topology, Geometry and holomorphic dynamics '' and GNSAGA INdAM.}
\maketitle
\section{Introduction}
Let $U$ be a compact connected semisimple Lie group and let $U^\C$ be its complexification \cite{akhiezer}. Let $\tau:U^\C \lra \mathrm{SL}(V)$
be an irreducible holomorphic representation. By the classical Borel-Weyl Theorem, the representation $\tau$ is completely determined by the unique
compact orbit of the $U^\C$-action on $\mathbb P(V)$ which is also the unique complex orbit of $U$.
The vector $v_{\mathrm{max}} \in V$ such that $x_0=[v_{\mathrm{max}}]\in \mathbb P(V)$ satisfies $U^\C \cdot x_0$ is compact, is the maximal weight vector
\cite{huckleberry-introduction-DMV}.

The $U$-action on $\mathbb P(V)$ is Hamiltonian and so there exists a momentum map $\mu:\mathbb P(V) \lra \mathfrak u^*$, where $\mathfrak u^*$ is the dual of the Lie algebra of $U$. If $T\subset U$ is
a maximal torus with Lie algebra $\mathfrak t$, then $\mu_{\mathfrak t}:=\mu\circ i^*$, where $i^*$ is the dual of the natural inclusion $i:\mathfrak t \hookrightarrow \liu$,
is the momentum map for the $T$-action on $\mathbb P(V)$ \cite{kirwan} .
With respect to the $T^\C$-action on $V$, there exist a finite set $\Delta(V,\mathfrak t^\C)\subset (\mathfrak t^\C)^*$ called weights of $V$, so that
\[
V=\bigoplus_{\lambda \in \Delta(V,\mathfrak t^\C)} V_\lambda,
\]
where $V_\lambda=\{v\in V:\, \tau(H)v=\lambda(H)v,\, \mathrm{for\ any\ } H\in \mathfrak t^\C\}$. It is well-known that
$v_{\mathrm{max}}\in V_{\lambda_{\mathrm{\tau}}}$ and $\dim  V_{\lambda_\tau}=1$. The functional $\lambda_\tau$ is called highest weight and $V$ is uniquely determined by $\lambda_\tau$ \cite{helgason}. These data are determined by the momentum map.
Indeed, $\bigcup_{\lambda \in \Delta(V,\mathfrak t^\C)} \mathbb P(V_\lambda)$ is the fixed point set for the $T^\C$-action on $\mathbb P(V)$ and
$\mu_{\mathfrak t}(\mathbb P(V_\lambda))=\lambda_{\vert_{\mathfrak t}}$. By the Atiyah-Guillemin-Sternberg convexity Theorem \cite{atiyah-commuting,guillemin-sternberg-convexity-1},
the set of the extreme points of the polytope $\mu_{\mathfrak t}(\mathbb P(V))$ are weights of $V$. This means that the momentum map of the $T$-action on $\mathbb P(V)$
encodes roots and associated structures of the irreducible representation $\tau$. Moreover, the unique compact orbit of the $U^\C$-action on $\mathbb P(V)$ achieves the maximum of the norm square momentum map. Indeed,  if $x\in \mathbb P(V)$ satisfies
$
\parallel \mu(x) \parallel =\mathrm{Max}_{y\in \mathbb P(V)} \parallel \mu(y) \parallel,
$
then $U^\C \cdot x$ is compact and it is a $U$-orbit \cite{gp,heinzner-schwarz-stoetzel}. In this paper, we analyze projective representations of
real reductive Lie groups. The main tool is the gradient map.

Let $\rho:G \lra \mathrm{GL}(V)$ be a representation on a finite dimensional real vector space. We identify $G$ with  $\rho(G)\subset \mathrm{GL}(V)$
and we assume that $G$ is closed and it is closed under transpose. This means  there exists a scalar product $\scalo$ on $V$ such that
$
G=K\exp (\liep),
$
where $K=G\cap \mathrm{O}(V)$ and $\liep=\lieg \cap \mathrm{Sym}(V)$. Here $\mathrm{O}(V)$ denotes the orthogonal group with respect to $\scalo$, $\mathrm{Sym}(V)$ the set of symmetric endomorphisms of $V$ and $\lieg$ the Lie algebra of $G$. By a standard theorem the existence of
$\scalo$ is proved for a large class of linear representations of a real reductive algebraic groups \cite{jabo,rs}.  The scalar product $\scalo$ defines
a Kempf-Ness function for the $G$-action on $V$ which is the basic tool to study certain geometric properties of linear actions of reductive real algebraic
groups \cite{kempf-ness,rs}. Note that $G$ is self-adjoint since $G$ is invariant with respect to the map $g\mapsto g^T$, where $g^T$
is the transpose of $g$ with respect to  $\scalo$. We recall that a classical Theorem of Mostow \cite{mostow-self} claims that any real reductive algebraic
subgroup $G$ of $\mathrm{GL}(n,\R)$ is conjugate to a self-adjoint subgroup of $\mathrm{GL}(n,\R)$. In this paper we always assume that $G$ is a
real noncompact semisimple linear Lie group and both the representations $\rho:G  \lra \mathrm{SL}(V)$ and $\rho:G \lra \mathrm{SL}(V^\C)$ are irreducible. If $\rho:G \lra \mathrm{SL}(V^\C)$ is reducible, then there exists a linear complex structure $J$ on $V$ such that
\[
\rho(G)\subset \{A\in \mathrm{SL}(V):\, AJ=JA\},
\]
\cite{fh}. Therefore $(V,J)$ is a complex vector space and the $G$-action on $V$ preserves $J$. This case has been  extensively studied in \cite{bsg}.
We may also assume, up to conjugate, that $V=\R^n$ and $\scalo$ is the canonical scalar product. This means that $G$ is a compatible subgroup of
$\mathrm{SL}(n,\R)$.

Let $\mathfrak g=\mathfrak k \oplus \liep$ be the Cartan decomposition of the Lie algebra of $G$ induced by the Cartan decomposition of $\mathrm{SL}(n.\R)$. $G$ is a compatible subgroup of $\mathrm{SL}(n,\C)$ as well. Since $\liek\cap i\liep=\{0\}$, by \cite[Proposition 3.3]{heinz-stoezel}, the Zariski closure of $G$  in $\mathrm{SL}(n,\C)$ is given by $U^\C$, where $U$ is the connected, compact semisimple Lie subgroup of $\mathrm{SU}(n)$  whose Lie algebra is given by $\mathfrak u=\liek \oplus i\liep$.  In particular,  $G$ is a compatible real form of $U^\C$ and the $U^\C$-action on $\C^n$ is irreducible as well.  Our aim is to investigate the natural irreducible
projective representation $\rho:G \lra \mathrm{PGL}(\R^n)$.  In the sequel we also denote by $\rho:G \lra \mathrm{PGL}(\C^n)$ and by
$\rho^\C:U^\C \lra \mathrm{PGL}(\C^n)$.

Let  $\scalo'$  be an $\mathrm{Ad}(\mathrm{SU}(n))$-invariant scalar product on $\mathfrak{su}(n)$. The $U$-action on $\mathbb P (\C^n)$ is Hamiltonian with momentum map
\[
\mu:\mathbb P(\C^n) \lra \liu, \qquad \mu=\pi_\liu \circ \Phi,
\]
where $\Phi(z)=-\frac{i}{2} \left(\frac{zz^*}{\parallel z \parallel^2} -\frac{1}{n}\mathrm{Id}_n\right)$ is the momentum map of the $\mathrm{SU}(n)$-action on
$\mathbb P(\C^n)$ and $\pi_\liu$ is the orthogonal projection of $\mathrm{su}(u)$ onto  $\liu$ \cite{kirwan}. The $G$-gradient map is defined as follows.

We also denote by $\scalo'$ the $\mathrm{Ad(SU(n))}$-invariant scalar product on
$i\mathfrak{su}(n)$  requiring that $i$ is an isometry of $\mathfrak{su}(n)$ onto $i\mathfrak{su}(n)$. Since $G$ is compatible, we have the subspace
$\liep\subset i\mathfrak u$. The $G$-gradient map $\mup:\mathbb P(\C^n) \lra \liep$  is the orthogonal projection of $i\mu$ onto $\liep$.  In other world we require
\[
\langle \mup(z),\beta \rangle'=\langle i\mu(z),\beta \rangle'=\langle \mu(z),-i\beta\rangle',
\]
for any $\beta \in \liep$. If $\lia \subset \liep$ is an Abelian subalgebra, then
\[
\mua:\mathbb P(\C^n) \lra \lia, \qquad \mua:=\pi_\lia \circ \mup,
\]
where $\pi_\lia$ is the orthogonal projection of $\liep$ onto $\lia$, is the $A=\exp(\lia)$-gradient map.
We usually restrict both $\mua$ and $\mup$ on $\mathbb P( \R^n)$.
By Borel-Weyl Theorem there exists a unique compact orbit $\mathcal O'$ of the $U^\C$-action on
$\mathbb P(\C^n)$. By a Theorem of Wolf \cite{wolf}, the $G$-action on $\OO'$ admits a unique compact orbit. We prove the following result.
\begin{teo}
The set $\mathcal O=\mathbb P(\R^n) \cap \mathcal O'$ is the unique compact $G$-orbit contained in $\mathcal O'$.
Moreover, $\mathcal O$ is a Lagrangian submanifold of $\mathcal O'$ and the fixed point set of an anti-holomorphic involutive isometry of $\mathcal O'$
induced by a complex conjugation of $\mathbb P(\C^n)$. In particular $\mathcal O$ is a totally geodesic submanifold of $\mathcal O'$.
\end{teo}
The vice-versa holds as well.

Let $\rho:U^\C \lra \mathrm{SL}(n,\C)$ be an irreducible holomorphic representation of a semisimple complex Lie group.
Let $G$ be a noncompact real form of $U^\C$. Let $\OO'$ be the unique compact orbit of the $U^\C$-action on $\mathbb P(\C^n)$ and let $\OO$ the unique compact $G$-orbit in $\OO'$.
\begin{teo}
If there exists an anti-holomorphic involution $T$ of $\mathbb P(\C^n)$ preserving $\OO'$ and such that $\OO$ is contained in the fixed point set of $T$, then  there exists a real subspace $V\subset \C^n$ such that $G$ acts irreducibly on $V$ and $V^\C=\C^n$.
\end{teo}
Given $\beta \in \liep$, we consider the parabolic subgroup
\[G^{\beta+} :=\{g \in G : \lim_{t\to - \infty} \exp({t\beta}) g
  \exp({-t\beta}) \text { exists} \}\\
\]
 of $G$ and $\mup^\beta:\mathbb P (\R^n) \lra \R$, where $\mup^\beta (z)=\langle \mup(z),\beta\rangle'$. Since $\beta \in \mathrm{Sym}_0 (n)$, $\beta$
 can be diagonalize.
Let $\lambda_1>\dots> \lambda_k$ be the eigenvalues of $\beta$.  We denote by $V_1,\dots, V_k$ the corresponding eigenspaces.
In view of the orthogonal decompositions $\R^n=V_1\oplus\dots\oplus V_{k}$, we get
\[
\mup^\beta ([x_1+\cdots+x_k])=
\frac{\lambda_1 \parallel x_1 \parallel^2 + \cdots +\lambda_k \parallel x_k \parallel^2}{\parallel x_1 \parallel^2 + \cdots +\parallel x_k \parallel^2}.
\]
In particular $\mathrm{Max}_{\mathbb P (\R^n)} (\beta)=\{p\in \mathbb P(\R^n):\, \mup^\beta (p)=\mathrm{max}_{z\in \mathbb P(\R^n)} \mup^\beta  \}=\mathbb P(V_1)$.
By Proposition \ref{rm}, we have $\mathrm{Max}_{\mathcal O} (\beta)\subset \mathrm{Max}_{\mathbb P (\R^n)} (\beta)$ and
$\mathrm{Max}_{\mathcal O} (\beta)\subset \mathrm{Max}_{\mathcal O'} (\beta)$. We point out that $\mathrm{Max}_{\mathcal O'} (\beta)$
is the unique compact orbit of the $(U^{\C})^{\beta+}$-action on $\mathbb P(\C^n)$ \cite[Corollary 1.0.1]{bj}.
\begin{teo}\label{teo1}
In the above setting, the following results hold true:
\begin{enumerate}
\item $V_1$ is the unique subspace of $\R^n$ such that $G^{\beta+}$ acts irreducibly on it;
\item $\maxm{\beta}{\mathcal O}$ is connected and it coincides with the unique compact orbit of the $G^{\beta+}$-action on $\mathcal O$. $\maxm{\beta}{\mathcal O}$  completely characterizes $V_1$;
\item $\maxm{\beta}{\mathcal O}$ is a Lagrangian submanifold of $\maxm{\beta}{\mathcal O'}$ and the fixed point set  of an anti-holomorphic involutive isometry  of $\maxm{\beta}{\mathcal O'}$. In particular, $\maxm{\beta}{\mathcal O}$ is  a totally geodesic submanifold of $\maxm{\beta}{\mathcal O'}$.
\end{enumerate}
\end{teo}
A compact orbit of the $G^{\beta+}$-action on $\OO'$ is contained in a compact orbit of $G$ due to the fact that $G=KG^{\beta+}$  \cite{borel-ji-libro}, see also Proposition \ref{decomposition-parabolic}.
Since $G$ has a unique compact orbit on $\OO'$, it follows that $\maxm{\beta}{\mathcal O}$ is the unique compact orbit of the $G^{\beta+}$-action on $\OO'$ as well.

Let $\mathcal E=\mathrm{conv}(\mup(\mathcal O))$. Since $\mathcal O$ is a compact $G$-orbit, it is a $K$-orbit \cite{heinzner-schwarz-stoetzel}.
Therefore, keeping in mind that the gradient map is $K$-equivariant, $\mathcal E$ is the convex hull of a $K$-orbit in $\liep$ and so it is a
polar orbitope \cite{biliotti-ghigi-heinzner-2,orbitope}.
In particular, any face of $\mathcal E$ is exposed \cite{biliotti-ghigi-heinzner-2}.

Let $F$ be a face of $\mathcal E$. By Lemma \ref{face-chain}, there exists a chain of faces
\[
F=F_0 \subsetneq F_1 \subsetneq \cdots F_k \subsetneq \mathcal E.
\]
Since any face is exposed, there exists $\beta_0,\beta_1,\ldots,\beta_k \in \liep$ such that
\[
F_i=\mathrm{Max}_{\mathcal E} (\beta_i):=\{z\in \mathcal E:\, \langle z , \beta_i \rangle =\mathrm{max}_{y\in \mathcal E} \langle y, \beta_i \rangle \}
\]
Applying Theorem \ref{teo1}, we get the following result.
\begin{prop}\label{chain}
Given a chain of faces $F=F_0 \subsetneq F_1 \subsetneq \cdots F_k \subsetneq \mathcal E$, there exist two chains of submanifolds
\[
\begin{array}{ccccccccc}
\maxm{\beta_0}{\mathcal O'} & \subsetneqq & \maxm{\beta_1}{\mathcal O'} & \subsetneqq & \cdots & \cdots & \subsetneqq &  \maxm{\beta_k}{\mathcal O'} & \subseteq \mathbb P(\C^n) \\
\cup                                &                      & \cup                                &                      &           &            &         & \cup & \cup \\
\maxm{\beta_0}{\mathcal O} & \subsetneqq & \maxm{\beta_1}{\mathcal O} & \subsetneqq & \cdots & \cdots & \subsetneqq &  \maxm{\beta_k}{\mathcal O} & \subseteq \mathbb P(\R^n)\\
\end{array}
\]
such that the vertical inclusions are Lagrangian and totally geodesic immersions.
\end{prop}
In \cite{biliotti-ghigi-heinzner-2}, the authors proved that the face structure of $\mathcal E$, up to $K$-equivalence, is completely determined by the face structure of $P=\mathcal E \cap \lia=\mua(\mathcal O)$, where $\lia\subset \liep$ is a maximal Abelian subalgebra, up to $\mathcal W=N_k (\lia)=\{k\in K:\, \mathrm{Ad}(k)(\lia)=\lia\}$-equivalence. We recall that $\mathcal W$ is called the Weyl group and it acts isometrically on $\lia$ as a finite group \cite{knapp}. By a Theorem of Kostant \cite{kostant-convexity}, $P$ is the convex hull of a Weyl group orbit and so it is a polytope \cite{schneider-convex-bodies}. By Proposition \ref{image}, $\mua(\mathbb P(\R^n))=\mua(\OO)$.
The centralizer of $\lia$ in $\lieg$ is given by
\[
\mathfrak z (\lia)=\mathfrak m \oplus \lia,
\]
where $\mathfrak m = \mathfrak z (\lia)\cap \mathfrak k$. If $\lia'\subset \mathfrak m$ is a maximal Abelian subalgebra of $\mathfrak m$, then $\lia'+ i \lia \subset \mathfrak u=\liek + i \liep$ is a maximal Abelian subalgebra of $\liu$ and so $(\lia'+ i \lia)^\C \subset \lieg^{\C}=\liu^{\C}$ is a Cartan subalgebra. Given $\lia, \lia'$,  and $\Pi\subset \Delta(\lieg,\lia)$ be a basis one can choose
a basis of $(\lia + i\lia')^*$ adapted to $\Pi$  and $(i\lia')^*$, see \cite[$p. 51-52$]{gjt},\cite[$p. 272-273$]{helgason}).
Let $\tilde \mu_\rho$ the highest weight of $\lieg^{\C}$ with respect to the partial ordering determined $\hat{\Delta}$. Let $x_o=[v_\rho]$, where $v_\rho$ is any highest weight vector.  It is well-known that
\[
\mu:\mathbb P(\C^n) \lra \lia'\oplus i \lia, \qquad \langle \mu(x_o) , \xi \rangle =\tilde{\mu_\rho}(\xi),
\]
see \cite{biliotti-ghigi-American} that has opposite sign convection for $\mu$, and \cite{baston-eastwood}.
By Proposition \ref{closed-orbit}, one can choose $v_\rho \in \R^n$. Moreover, $G\cdot x_o=\mathcal O$ and so
\[
\langle \mua(x_o),\xi \rangle=(i \tilde \mu_\rho)(\xi).
\]
$(i\tilde{\mu_\tau})_{\vert{\lia}}$ is the highest weight of $\lieg$ with respect the induced order on $\lia^*$. In the sequel we denote by $\mu_\rho=(i\tilde{\mu_\rho})_{\vert{\lia}}$ and also its dual in $\lia$ with respect to the scalar product $\scalo$.
\begin{prop}\label{weight}
$P=\mathrm{conv}(\mathcal W \cdot \mu_\rho)$.  In particular the weights of $\rho$ are contained in the convex hull of the Weyl group orbit through the highest weight
$\mu_\rho$.
\end{prop}
Let $\beta \in \liep$. Then
\[
F_\beta (\mathcal E)=\{p\in \mathcal E:\, \langle p, \beta \rangle = \mathrm{max}_{y\in \mathcal E} \langle y , \beta \rangle\},
\]
is a face of $\mathcal E$ and any face  of $\mathcal E$ is given by $F_\beta (\mathcal E)$ for some $\beta \in \liep$.  Since $\mup(\OO)$ is a $K$-orbit, it
is a fundamental fact that the action of K on $\mup(\OO)$  extends to an action of G see \cite[Lemma 5]{heinzner-stoetzel-global}.
The set of extreme points of $F_\beta (\mathcal E)$, that we denote by $\mathrm{ext}\, (F_\beta (\mathcal E))$, is contained in $\mup(\OO)$. If $V_1$ is the eigenspace of $\beta$ relative to the maximal eigenvalue, then
$\mathbb P(V_1)=\mathrm{Max}_{\mathbb P (\R^n)} (\beta)=\mup^{-1}( F_\beta (\mathcal E))$. We prove that
\[
\{g\in G:\, gV_1=V_1\}=Q(F_\beta (\mathcal E))=\{h\in G:\, h \mathrm{ext}\,  F_\beta (\mathcal E)=\mathrm{ext}\, F_\beta (\mathcal E)\},
\]
is a parabolic subgroup of $G$ which it contains $G^{\beta+}$. Moreover, the $Q(F_\beta (\mathcal E))$-action on $V_1$ is irreducibly  and
$\maxm{\beta}{\mathcal O}$ is  the unique compact orbit of the $Q(F_{\beta}(\mathcal E))$-action on $\OO$.

The group $K$ acts on the set of faces of $\mathcal E$. Up to this $K$-action, a face of  $\mathcal E$ is described
in terms of root data \cite{biliotti-ghigi-heinzner-2}. The main tool is the notion of $\mu_\rho$-connected subset of the set of positive roots
$\Pi$ with respect to a fixed order. This notion was introduced  by Satake
in the study of the boundary components of the Satake compactifications of a symmetric space of noncompact type  \cite{Satake}.

A subset $I\subset \simple$ is $\mu_\rho$-\enf{connected} if
$I\cup\{\mu_\rho \}$ is connected, i.e., it is not the union of subsets orthogonal with respect to the Killing form. We denote by $I'$ the collection of all simple roots orthogonal to $\{\mu_\rho \}\cup I$. The set $J:=I\cup I'$ is called the $\mu_\rho$-\enf{saturation} of $I$. Given $I$, we consider the standard parabolic subalgebras $\mathfrak q_I$ and $\mathfrak q_J$, respectively, see Section \ref{compatible-parabolic} and \cite{borel-ji-libro,Satake}. We denote by $Q_I$, respectively $Q_J$, the parabolic subgroup of $G$ with Lie algebra $\mathfrak q_I$, respectively with Lie algebra $\mathfrak q_J$. If $Q_I=G^{\beta+}$ then $Q_J= Q(F_\beta (\mathcal E))$ \cite{biliotti-ghigi-heinzner-2} and  both $Q_I$ and $Q_J$ act irreducibly on the eigenspace of $\beta$ associated to the  maximum eigenvalue, that we denote by $W_I$. Satake proved that $W_I$ can be defined in terms of root data \cite[Lemma I.4.25, p. 69]{borel-ji-libro} and the $Q_J$-action on $W_I$ is completely determined by the $Q_I$-action on $W_I$.

Let $Q$ be a parabolic subgroup and let $W$ be the unique subspace of $V$ such that $Q$ acts irreducibly on it. We show  there exists $k\in K$ such that $kW=W_I$ and
\[
Q_I \subseteq k Q k^{-1}\subseteq Q_J.
\]
This means that up to the $K$-action, boundary components of $\mathcal E$ completely describe the irreducible representations of the
parabolic subgroups of $G$ induced by $\rho:G \lra \mathrm{SL}(n,\R)$.  We also prove that irreducible representations of parabolic
subgroups of $G$ induced by $\rho:G \lra \mathrm{SL}(n,\C)$ are the complexification of the irreducible representations of parabolic
subgroups of $G$ induced by $\rho:G \lra \mathrm{SL}(n,\R)$. Hence, we get the following result.
\begin{teo}
The face structure of $\mathcal E$, up to $K$-equivalence, describes the irreducible representations of the parabolic subgroups of $G$ induced
by $\rho:G \lra \mathrm{SL}(n,\R)$
and so induced also by $\rho:G \lra \mathrm{SL}(n,\C)$.
\end{teo}
In the sequel we always refer to \cite[Section I.1 Real Parabolic subgroups]{borel-ji-libro} and Section \ref{compatible-parabolic}.

Let $I\subset \Pi$ be a $\mu_\rho$-connected subset. By Theorem \ref{parabolici-faccie}, see \cite{biliotti-ghigi-heinzner-2}, $Q_I \cdot \mu_\rho$ is the set of extreme points of a face of  $\mathcal E$ that we denote by $F_I$. Although the $G$-gradient map is not $G$-equivariant, by Proposition \ref{gi} we get $Q_I \cdot x_o$ is compact and
\[
\mup(Q_I \cdot x_o )=Q_I \cdot \mu_\rho.
\]
Let $\mathfrak a_I :=\bigcap_{\alpha \in I} \, \mathrm{Ker}\, \alpha$ and let $\mathfrak a^I$  be the orthogonal complement of $\mathfrak a_I$ in $\mathfrak a$.  Then
$
\mathfrak q_I=\mathfrak n_I \oplus \mathfrak a_I \oplus \mathfrak m_I,
$
where $\mathfrak m_I=\mathfrak z_{\mathfrak k}(\mathfrak a)\oplus \mathfrak a^I \bigoplus_{\alpha \in I} \mathfrak g_\alpha$ is the Lie algebra of a Levi factor of $Q_I$, that we denote by $M_I$ which is not connected in general. We recall that $\mathfrak g_{\alpha}:=\{v\in \lieg:\, [H,v]=\alpha(H)v,\, \forall H\in \lia\}$. $M_I$ is compatible and $K_I=K\cap Q_I$  is a maximal compact subgroup of $M_I$. The Abelian subalgebra $\mathfrak a^I$ is a maximal Abelian subalgebra of $\mathfrak m_I\cap \liep$.
Let $\mathcal W_I =N_{K_I}(\lia_I)$. $\mathcal W_I$ is the subgroup of $\mathcal W$ generated by the root reflections defined by the element of $I$.
We split $\mu_\rho=y_0+y_1$, where $y_0 \in \mathfrak a_I$ and $y_1\in \mathfrak a^I$.
\begin{teo}
The map $I \mapsto \mathrm{conv}\, (\mathcal W_I \cdot \mu_\rho)$  induces a bijection between the $\mu_\rho$-connected subset of $\Pi$ and the faces of $P$ up to the Weyl-group action. Moreover,
\[
\mua(Q_I \cdot x_o)=\mathrm{conv}\, (\mathcal W_I \cdot \mu_\rho)=  y_0+\mathrm{conv}\, (\mathcal W_I \cdot y_1).
\]
\end{teo}
The description of the faces of $P$  is proved in \cite{biliotti-ghigi-heinzner-2}, see \cite[Theorem 6.2]{ks}. The above statement is quoted from Casselman \cite[Theorem 3.1]{casselman}, where it is proved in the more general context of
of arbitrary finite Coxeter groups. Casselman also pointed out that the result is already implicit in the papers \cite{Satake} and \cite{borel-tits}. Our proof uses the description of the faces of $P$ given in \cite{biliotti-ghigi-heinzner-2}  and the techniques of the $G$-gradient map.

We also investigate the norm square gradient map and the norm square momentum map.

Let
\[
\nu_\liep: \mathbb P(\R^n) \lra \R, \qquad p\mapsto \frac{1}{2}\parallel \mup (p) \parallel^2,
\]
denote the norm square gradient map and let
\[
\nu_\liu (p):\mathbb P(\C^n) \lra \R, \qquad p \mapsto =\frac{1}{2}\parallel \mu (p) \parallel^2,
\]
denote the norm square momentum map. The gradient of $\nu_\liep$ with respect to the Riemannian metric induced by the Fubini-Study metric on
$\mathbb P(\C^n)$ is given by $\mathrm{grad}\, \nu_\liep (x)= \mu(x)_{\mathbb P (\R^n)}$, where
$\mu(x)_{\mathbb P (\R^n)}:=\desudtzero \exp(t\mup(x))x$ \cite{heinzner-schwarz-stoetzel}.
The gradient of $\nu_\liu$ is given by $\mu(x)_{\mathbb P (\C^n)}:=\desudtzero \exp(t\mu(x))x$ \cite{kirwan}. We point out that $x\in \mathbb P(\R^n)$ is a critical point of $\nu_\liep$ if and only if $x\in \mathbb P(\R^n)$ is a critical point of $\nu_\liu$, see also \cite{jabo}, and
\[
\mathrm{Max}_{x\in \mathbb P(\R^n)} \nu_\liep =\mathrm{Max}_{x\in \mathbb P(\C^n)} \nu_\liu, \qquad
\mathrm{Inf}_{x\in \mathbb P(\R^n)} \nu_\liep =\mathrm{Inf}_{x\in \mathbb P(\C^n)} \nu_\liu.
\]
The negative gradient flow line of $\nu_\liep$ through $x_0\in \mathbb P(\R^n)$ is the solution of the differential equation
\[ \left\{ \begin{array}{ll}
\dot{x}(t) = -\beta_{\mathbb P (\R^n)} (x(t)), \quad t\in \mathbb{R} \\
 x(0) = x_0.\end{array} \right.
\]
It s defined for any $t\in \R$ and the limit $$x_\infty := \lim_{t \rightarrow \infty} x(t) $$ exist \cite[Theorem 3.3]{bjmain}, see also \cite{grs} for the complex case.
Applying results proved in \cite[Theorem 4.7]{bjmain}, we get the following results.
\begin{prop}
Let  $x_0 \in \mathbb P(\R^n)$. Then
\[
\begin{split}
\parallel \mu_\mathfrak{p}(x_\infty)\parallel &=\text{Inf}_{g\in G}\parallel \mu_\mathfrak{p}(gx_0)\parallel \\ &=\text{Inf}_{g\in G^\C} \parallel \mu(gx_0)\parallel.
\end{split}
\]
Let $x_0\in \mathbb P(\C^n)$ be such that $U^\C \cdot x_0 \cap \mathbb P(\R^n) \neq \emptyset$. Then
\[
\mathrm{Inf}_{g\in U^\C}\parallel \mu (gx_0 )\parallel =\mathrm{Inf}_{y\in U^\C \cdot x_0 \cap \mathbb P(\R^n)} \parallel \mup (y)\parallel.
\]
\end{prop}
We also investigate the stratification of $\mathbb P(\R^n)$ with respect to $\nu_\liep$.
\begin{prop}
The norm square gradient map $\nu_\liep$ has a unique open stratum which is the minimal stratum. This stratum is open, dense and it is given by the intersection of $\mathbb P(\R^n)$ with the minimal stratum of  $\nu_\liu$. \end{prop}
In principle there could be many different open stratum for the norm square gradient map but we do not know any example.
Any such example would imply that the nonAbelian convexity Theorem fails \cite{heinzner-schutzdeller}.
\section{Preliminaries}
\subsection{Convex geometry}\label{convex-geometry}
It is useful to recall a few definitions and results regarding convex
 sets. The reader may refer for instance  to \cite{schneider-convex-bodies} for more details.

 Let $V$ be a real vector
 space with a scalar product $\scalo$ and let $E\subset V$ be a
 compact convex subset.  The \emph{relative interior} of
 $E$, denoted $\mathrm{relint} E$, is the interior of $E$ in its affine hull.
 A face $F$ of $E$ is a convex subset $F\subset E$ with the following
 property: if $x,y\in E$ and $\mathrm{relint}[x,y]\cap F\neq \emptyset$, then
 $[x,y]\subset F$.  The \emph{extreme points} of $E$ are the points
 $x\in E$ such that $\{x\}$ is a face. We denote by $\mathrm{ext}\, E$ the set of the extreme points of $E$. By a Theorem of Minkowski \cite[Corollary 1.4.5 p.19]{schneider-convex-bodies},
 $E$ is the convex hull of  its extremal points. Since  $E$ is compact the faces
 are closed \cite[p. 62]{schneider-convex-bodies}.  A face distinct
 from $E$ and $\vacuo$ will be called a \enf{proper face}.  The
 \emph{support function} of $E$ is the function $ h_E : V \lra \R$, $
 h_E(u) = \max_{x \in E} \langle x, u \rangle$.  If $ u \neq 0$, the
 hyperplane $H(E, u) : = \{ x\in E : \langle x, u \rangle = h_E(u)\}$ is
 called the \emph{supporting hyperplane} of $E$ for $u$. The set
   \begin{gather}
     \label{def-exposed}
     F_u (E) : = E \cap H(E,u)
   \end{gather}
   is a face and it is called the \emph{exposed face} of $E$ defined by
   $u$.
In general not all faces of a convex subset are exposed. A simple example is given by the convex hull of a closed disc and a point outside the disc:
the resulting convex set is the union of the disc and a triangle. The two vertices of the triangle that lie on the boundary of the disc are non-exposed faces
  \begin{lemma}
[\protect{\cite[Lemma 3]{biliotti-ghigi-heinzner-1}}]
\label{ext-facce}
If $F$ is a face of a convex set $E$, then $\ext F = F \cap \ext E$.
 \end{lemma}
 \begin{lemma}[\protect{\cite[Lemma 8]{biliotti-ghigi-heinzner-1}}]
   \label{face-chain}
   If $E$ is a compact convex set and $F\subset E$ is a face, then
   there is a chain of faces $ F_0=F \subsetneq F_1 \subsetneq \cds
   \subsetneq F_k=E $ which is maximal, in the sense that for any $i$
   there is no face of $E$ strictly contained between $F_{i-1}$ and
   $F_i$.
 \end{lemma}
\begin{lemma}[\protect{\cite[Prop. 5]{biliotti-ghigi-heinzner-1}}] \label{u-cono}
   If $F \subset E$ is an exposed face, the set $\CF : = \{ u\in V:
   F=F_u(E) \}$ is a convex cone. If $K$ is a compact subgroup of
   $O(V)$ that preserves both $E$ and $F$, then $\CF$ contains a fixed
   point of $K$.
 \end{lemma}
We denote by $\CF^K$ the elements of $\CF$ fixed by a compact group $K$.
The faces of a convex compact set give a stratification of $E$. The following result is well-known and a proof is given in \cite[p. 62]{schneider-convex-bodies}
\begin{teo} \label{schneider-facce} If $E$ is a compact convex set and
   $F_1,F_2$ are distinct faces of $E$, then $\relint F_1 \cap \relint
   F_2=\vacuo$. If $G$ is a nonempty convex subset of $ E$ which is
   open in its affine hull, then $G \subset\relint F$ for some face
   $F$ of $E$. Therefore $E$ is the disjoint union of the
     relative interiors of its faces.
 \end{teo}
The following result is probably well-known.
\begin{prop}\label{convex-criterium}
Let $C_1 \subseteq C_2$ be two compact convex subsets of $V$. Assume that for any $\beta \in V$ we have
\[
\mathrm{max}_{y\in C_1} \langle y , \beta \rangle=\mathrm{max}_{y\in C_2} \langle y , \beta \rangle.
\]
Then $C_1=C_2$.
\end{prop}
\begin{proof}
We may assume without loss of generality  that the affine hull of $C_2$ is $V$.
Assume by contradiction that $C_1 \subsetneq C_2$. Since $C_1$ and $C_2$ are both compact, it follows that there exists $p\in \partial C_1$ such that $p\in \stackrel{o}{C_2}$. Since every face of a compact convex set is contained in an exposed face \cite{schneider-convex-bodies}, there exists $\beta \in V$ such that
\[
\mathrm{max}_{y\in C_1} \langle y , \beta \rangle=\langle p, \beta \rangle.
\]
This means the linear function $x\mapsto \langle x, \beta \rangle$ restricted on $C_2$ achieves its maximum at an interior point which is a contradiction.
\end{proof}
Let $E$ be a $K$-invariant convex body of $\liep$. Let $\lia \subset \liep$ be a maximal Abelian subalgebra of $\liep$ and let $\mathcal W:=\{\mathrm{Ad}(k):\, k\in K\, \mathrm{and} \mathrm{Ad}(k)(\lia)=\lia\}$. $\mathcal W$ is called \emph{Weyl group} and it acts on $\lia$ as a finite group \cite{knapp}. Let $P=E\cap \lia$ and let $\pi_\lia:\liep \lra \lia$ be the orthogonal projection of $\liep$ onto $\lia$.
The following result is proved in \cite{gichev-polar}, see also \cite{ks}.
\begin{teo}\label{convex-reduction}
$P$ is a $\mathcal W$-invariant convex body of $\lia$ satisfying $\pi_\lia (E)=P$ and $E=KP$. Hence if $E_1,E_2 \subset \liep$ are two $K$-invariant convex bodies of $\liep$ then $E_1 \cap \lia =E_2 \cap \lia$ if and only if $E_1=E_2$.
\end{teo}
The $\mathcal W$-action on $P$ induces an action on the faces of $P$. Similarly $K$ acts on the set of faces of $E$. Denote these sets by $\mathscr F (P)$ respectively by $\mathscr F (E)$. Let $\sigma$  be a face of $P$. Then its affine hull is given by $x_o+W$, where $W\subset \liep$ is a subspace.  We denote by $\sigma^\perp$ the orthogonal of $W$. The following result is proved in \cite{bgh-israel-p}.
\begin{teo}\label{meo-israel}
The map $\mathscr F (P) \lra \mathscr F (E), \sigma \mapsto K^{\sigma^\perp} \cdot \sigma$ is well-defined and induces a bijection between $\mathscr F (P)/ \mathcal W$ and $\mathscr F (E)/K$. Moreover, $\sigma$ is an exposed face if and only if $K^{\sigma^\perp} \cdot \sigma$ does.
\end{teo}
\section{Compatible subgroups and parabolic subgroups}\label{compatible-parabolic}
In the sequel we always refer to \cite{borel-ji-libro,gjt,heinzner-schwarz-stoetzel,knapp}.

Let $U$ be compact connected Lie group. Let $U^\C$ be its universal complexification which
 is a linear reductive complex algebraic group \cite{akhiezer}. We
   denote by $\theta$ both the conjugation map $\theta : \liu^\C \lra
   \liu^\C$ and the corresponding group isomorphism $\theta : U^\C \lra
   U^\C$.  Let $f: U \times i\liu \lra U^\C$ be the diffeomorphism
 $f(g, \xi) = g \exp \xi$.  Let $G\subset U^\C$ be a closed
 subgroup. Set $K:=G\cap U$ and $\liep:= \lieg \cap i\liu$.  We say
 that $G$ is \emph{compatible} if $f (K \times \liep) = G$.  The
 restriction of $f$ to $K\times \liep$ is then a diffeomorphism onto
 $G$. Hence $\lieg=\liek\oplus \liep$ is the familiar Cartan decomposition  and so  $K$ is a maximal compact subgroup of $G$.  Note that $G$ has finitely many connected components. Since $U$ can be embedded in $\Gl(N,\C)$ for
 some $N$, and any such embedding induces a closed embedding of
 $U^\C$, any compatible subgroup is a closed linear group. Moreover
 $\lieg$ is a real reductive Lie algebra, hence $\lieg =
 \liez(\lieg)\oplus [\lieg, \lieg]$. Denote by $G_{ss}$ the analytic
 subgroup tangent to $[\lieg, \lieg]$. Then $G_{ss}$ is closed and
 $G^o=Z(G)^o \cd G_{ss}$ \cite[p. 442]{knapp}, where $G^o$, respectively $Z(G)^o$, denotes the connected component of  $G$, respectively of $Z(G)$, containing $e$.
The following lemma is well-known. A proof is given in \cite[pag.584]{biliotti-ghigi-heinzner-2}.
 \begin{lemma}$\, $ \label{lemcomp}
   \begin{enumerate}
   \item \label {lemcomp1} If $G\subset U^\C$ is a compatible
     subgroup, and $H\subset G$ is closed and $\theta$-invariant,
     then $H$ is compatible if and only if $H$ has only finitely many connected components.
   \item \label {lemcomp2} If $G\subset U^\C$ is a connected
     compatible subgroup, then $G_{ss}$ is compatible.
     \item \label{lemcomp3} If $G\subset U^\C$ is a compatible
       subgroup, and $E\subset \liep$ is any subset, then \[G^E=\{g\in G:\, \mathrm{Ad}(g)(z)=z,\, \forall z\in E\}\] is
       compatible. Indeed, $G^E=K^E \exp(\liep^E)$, where $K^E=G^E \cap K$ and $\mathfrak p^E=\{v\in \liep:\, [v,E]=0\}$.
   \end{enumerate}
 \end{lemma}
A subalgebra $\lieq \subset \lieg$ is \enf{parabolic} if
$\lieq^\C$ is a parabolic subalgebra of $\lieg^\C$.  One way to
describe the parabolic subalgebras of $\lieg$ is by means of
restricted roots.  If $\lia \subset \liep$ is a maximal subalgebra,
let $\roots(\lieg, \lia)$ be the (restricted) roots of $\lieg$ with
respect to $\lia$, let $\lieg_\la$ denote the root space corresponding
to $\lambda$ and let $\lieg_0 = \liem \oplus \lia$, where $\liem =
\liez_\liek(\lia)=\liez (\lia)\cap \liek$. We denote by $\mathfrak z (\lia)=\{x\in \lieg:\, [x,\lia]=0\}$.  Let $\simple \subset \roots(\lieg, \lia)$ be a
base and let $\roots_+$ be the set of positive roots. If $I\subset
\simple$, set $\roots_I : = \spam(I) \cap \roots$. Then
\begin{gather}
  \label{para-dec}
  \lieq_I:= \lieg_0 \oplus \bigoplus_{\la \in \roots_I \cup \roots_+}
  \lieg_\la
\end{gather}
is a parabolic subalgebra. Conversely, if $\lieq \subset \lieg$ is a
parabolic subalgebra, then there are a maximal subalgebra $\lia
\subset \liep$ contained in $\lieq$, a base $\simple \subset
\roots(\lieg, \lia)$ and a subset $I\subset \simple $ such that $\lieq
= \lieq_I$.  We can further introduce
\begin{gather}
\label{notaz-I}
\begin{gathered}
  \lia_I : = \bigcap_{\la \in I} \ker \la \qquad \lia^I := \lia_I^\perp \\
  \lien_I = \bigoplus_{\la \in \roots_+ \setminus \roots_I} \lieg_\la
  \qquad \liem_I : = \liem \oplus \lia^I \oplus \bigoplus_{\la \in
    \roots_I}\lieg_\la.
\end{gathered}
\end{gather}
Then $\lieq_I = \liem_I \oplus \lia_I \oplus \lien_I$. Since
$\theta\lieg _ \la = \lieg_{-\la}$, it follows that $ \lieq_I \cap
\theta\lieq_I = \lia_I \oplus \liem_I$.  This latter Lie algebra coincides
with the centralizer of $\lia_I$ in $\lieg$. It is a Levi factor
of $\lieq_I$ and
\begin{gather}
  \label{liaI}
  \lia_I =\liez (\lieq_I \cap \theta\lieq_I) \cap \liep.
\end{gather}
If we denote by $\Delta_{-}$ the set of negative root, then $\mathfrak n_I^{-}= \bigoplus_{\la \in \roots_{-} \setminus \roots_I} \lieg_\la$ is a subalgebra.
It follows from standard commutation relations that $\mathfrak z (\lia_I)$ normalizes $\mathfrak n_I$ and $\mathfrak n_I^{-}$ and the centralizer of $\lia^I$ in either is reduced to zero. Then, keeping in mind $\lieg=\mathfrak n_I^{-}\oplus \lieq_I$, $\lieq_I$ is self-normalizing.
\begin{defin}
A subgroup $Q$ of $G$ is called parabolic if it is the normalizer of a parabolic subalgebra in $\lieg$.
\end{defin}
The normalizer of $\lieq_I$ is the \emph{standard parabolic subalgebra} $Q_I$. Let $R_I$ and let $A_I$ be the unique connected Lie subgroups of $G$ with Lie algebra equals to $\mathfrak n_I$ and $\lia_I$ respectively. $R_I$ is the unipotent radical of  $Q_I$. The group $Q_I$ is the semidirect product of $R_I$ and of $Z(A_I)$, i.e., the centralizer of $A_I=\exp(a_I)$ in $G$. Moreover, $Z(A_I)=A_I \times M_I$, where $M_I$ is a closed Lie group whose Lie algebra is $\mathfrak m_I$. It is not connected in general but it is compatible. Since $M_I$ is stable with respect to the Cartan involution, $K_I=M_I\cap K$ is  maximal compact in $M_I$. It is also maximal compact in $Q_I$ and the quotient
\[
X_I=M_I/K_I=Q_I/ K_I A_I N_I
\]
is a symmetric space of noncompact type for $M_I$. Finally, as a consequence of the Iwasawa decomposition $G=NAK$, where $N=\exp(\mathfrak n)$, $\mathfrak n=\mathfrak n_{\emptyset}$, and $NA\subset Q_I$, we get the following result
\begin{prop}\label{decomposition-parabolic}
$G=KQ_I$.
\end{prop}
Another way to describe parabolic subalgebras of $\lieg$ is the
following.  If $\beta \in \liep$, the endomorphism $\ad \beta \in \End
\lieg$ is diagonalizable over $\R$. Denote by $ V_\la (\ad \beta) $ the
eigenspace of $\ad \beta$ corresponding to the eigenvalue $\la$.  Set
\begin{gather*}
  \lieg^{\beta+}: = \bigoplus_{\la \geq 0} V_\la (\ad \beta).
\end{gather*}
The following result is proved in \cite{biliotti-ghigi-heinzner-2}.
\begin{lemma}
  \label{para-beta}
  For any $\beta$ in $ \liep$, $\lieg^{\beta+}$ is a parabolic
  subalgebra of $\lieg$.  If $\lieq \subset \lieg$ is a parabolic
  subalgebra, there is some vector $\beta \in \liep$ such that $\lieq
  = \lieg^{\beta+}$.  The set of all such vectors is an open convex
  cone in
$\liez(\lieq \cap \theta\lieq) \cap \liep$.
\end{lemma}

A \enf {parabolic subgroup} of $G$ is a subgroup of the form $Q =
N_G(\lieq)$ where $\lieq $ is a parabolic subalgebra of $\lieg$.
Equivalently, a parabolic subgroup of $G$ is a subgroup of the form
$\tilde Q \cap G$ where $\tilde Q$ is a parabolic subgroup of $G^\C$ and $\tilde{\mathfrak q}$ is the
complexification of a subspace $\lieq \subset \lieg$.

If $\beta
\in \liep$ set
\begin{gather*}
\begin{gathered}
  G^{\beta+} :=\{g \in G : \lim_{t\to - \infty} \exp({t\beta}) g
  \exp({-t\beta}) \text { exists} \}\\
  R^{\beta+} :=\{g \in G : \lim_{t\to - \infty} \exp({t\beta})
  g \exp({-t\beta}) =e \} \\
\end{gathered}
\qquad
  \mathfrak r^{\beta+}: = \bigoplus_{\la > 0} V_\la (\ad \beta).
\end{gather*}
The following result characterizes completely the parabolic subgroups of $G$. The result is classical and a proof is given in \cite{biliotti-ghigi-heinzner-2}.
\begin{lemma}
  $G^{\beta +} $ is a parabolic subgroup of $G$ with Lie algebra
  $\lieg^{\beta+}$ and it is the semidirect product of $G^{\beta}$ with $R^{\beta+}$. Moreover, $G^\beta$ is a Levi factor, $R^{\beta+}$ is connected and it is the
  unipotent radical of $G^{\beta+}$. Finally, every parabolic subgroup of $G$ equals
  $G^{\beta+}$ for some $\beta \in \liep$.
\end{lemma}

\subsection{Basic properties of the gradient map}\label{subsection-gradient-moment}
Let $(Z, \omega)$ be a K\"ahler manifold. Assume that $U^\C$ acts
holomorphically on $Z$, that $U$ preserves $\om$ and that there is a
momentum map $\mu: Z \lra \liu$.  If $\xi \in \liu$ we denote by $\xi_Z$
the induced vector field on $Z$, i.e., $\xi_Z (p)=\desudtzero \exp(t\xi) p$, and we let $\mu^\xi \in C^\infty (Z)$ be
the function $\mu^\xi(z) := \langle \mu(z),\xi\rangle$, where $\langle \cdot,\cdot \rangle$ is an $\Ad(U)$-invariant scalar product on $\liu$.  That $\mu$ is the
momentum map means that it is $U$-equivariant and that $d\mu^\xi =
i_{\xi_Z} \omega$.

Let $G \subset U^\C$ be compatible.
If $z \in Z$, let $\mup (z) \in \liep$ denote $-i$ times the component
of $\mu(z)$ in the direction of $i\liep$.  In other words, if we also denote by $\langle \cdot, \cdot\rangle$ the $\Ad(U)$-invariant scalar product on $i\liu$  requiring the multiplication by $i$ is an isometry of $\liu$ onto $i\liu$, then
$\langle \mup (z) , \beta \rangle = \langle i\mu(z) , \beta\rangle=\langle \mu(z),-i\beta\rangle$ for any
$\beta \in \liep$, defines the $G$-\emph{gradient map}
\begin{gather*}
  \mu_\liep : Z \lra \liep.
\end{gather*}
Let $\mup^\beta \in C^\infty (Z)$ be the function $ \mup^\beta(z) = \langle
\mup(z) , \beta\rangle = \mu^{-i\beta}(z)$.  Let $(\cdot,\cdot)$ be the K\"ahler
metric associated to $\om$, i.e. $(v, w) = \om (v, Jw)$. Then
$\beta_Z$ is the gradient of $\mup^\beta$. If $M \subset Z$ is a
locally closed $G$-invariant submanifold, then $\beta_M $ is the
gradient of $\mup^\beta \restr{M}$ with respect to the induced
Riemannian structure on $M$. From now on we always assume that $M$ is compact and connected.
\begin{teo}\label{line}[Slice Theorem \protect{\cite[Thm. 3.1]{heinzner-schwarz-stoetzel}}]
  If $x \in M$ and $\mup(x) = 0$, there are a $G_x$-invariant
  decomposition $T_x M = \lieg \cd x \oplus W$, open $G_x$-invariant
  subsets $S \subset W$, $\Omega \subset M$ and a $G$-equivariant
  diffeomorphism $\Psi : G \times^{G_x}S \ra \Omega$, such that $0\in
  S, x\in \Omega$ and $\Psi ([e, 0]) =x$.
\end{teo}
Here $G \times^{G_x}S$ denotes the associated bundle with principal
bundle $G \ra G/G_x$.
\begin{cor} \label{slice-cor} If $x \in M$ and $\mup(x) = \beta$,
  there are a $G^\beta$-invariant decomposition $T_x M = \lieg^\beta
  \cd x \, \oplus W$, open $G^\beta$-invariant subsets $S \subset W$,
  $\Omega \subset M$ and a $G^\beta$-equivariant diffeomorphism $\Psi
  : G^\beta \times^{G_x}S \ra \Omega$, such that $0\in S, x\in \Omega$
  and $\Psi ([e, 0]) =x$.
\end{cor}
This follows applying the previous theorem to the action of $G^\beta$
with the gradient map $\widehat{\mu_{\liu^\beta}} := \mu_{\liu^\beta} -
i\beta$, where $\mu_{\liu^\beta}$ denotes the projection of $\mu$ onto
$\liu^\beta$.
See \cite[p.$169$]{heinzner-schwarz-stoetzel} and \cite{sjamaar} for more details.

If $\beta \in \liep$, then $\beta_M$ is a vector field on
$M$, i.e. a section of $TM$. For $x\in M$, the differential is a map
$T_x M \ra T_{\beta_M (x)}(TM)$. If $\beta_M (x) =0$, there is a
  canonical splitting $T_{\beta_M (x)}(TM) = T_x M \oplus
  T_x M$. Accordingly $d\beta_M (x)$ splits into a horizontal and a
vertical part. The horizontal part is the identity map. We denote the
vertical part by $d\beta_M  (x)$.  It belongs to $\End(T_x M)$.  Let
$\{\phi_t=\exp(t\beta)\} $ be the flow of $\beta_M$.  There
  is a corresponding flow on $TM$. Since $\phi_t(x)=x$, the flow on
  $TM$ preserves $T_x M$ and there it is given by $d\phi_t(x) \in
  \Gl(T_x M)$.  Thus we get a linear $\R$-action on $T_x M$ with
  infinitesimal generator $d\beta_M  (x) $.
\begin{cor}
  \label{slice-cor-2}
  If $\beta \in \liep $ and $x \in M$ is a critical point of $\mupb$,
  then there are open invariant neighborhoods $S \subset T_x M$ and
  $\Omega \subset M$ and an $\R$-equivariant diffeomorphism $\Psi : S
  \ra \Omega$, such that $0\in S, x\in \Omega$, $\Psi ( 0) =x$. Here
  $t\in \R$ acts as $d\phi_t(x)$ on $S$ and as $\phi_t$ on $\Omega$.)
\end{cor}
\begin{proof}
  The subgroup $H:=\exp(\R \beta)$ is compatible.  It is enough to
  apply the previous corollary to the $H$-action at $x$.
\end{proof}
Let $x \in
\Crit(\mu^\beta_\liep)=\{y\in M:\, \beta_M (y)=0\}$. Let $D^2\mup^\beta(x) $ denote the Hessian,
which is a symmetric operator on $T_x M$ such that
\begin{gather*}
  ( D^2 \mup^\beta(x) v, v) = \frac{\mathrm d^2}{\mathrm
    dt^2} 
  (\mup^\beta\circ \ga)(0)
\end{gather*}
where $\ga$ is a smooth curve, $\ga(0) = x$ and $ \dot{\ga}(0)=v$.
Denote by $V_-$ (respectively $V_+$) the sum of the eigenspaces of the
Hessian of $\mupb$ corresponding to negative (resp. positive)
eigenvalues. Denote by $V_0$ the kernel.  Since the Hessian is
symmetric we get an orthogonal decomposition
\begin{gather}
  \label{Dec-tangente}
  T_x M = V_- \oplus V_0 \oplus V_+.
\end{gather}
Let $\alfa : G \ra M$ be the orbit map: $\alfa(g) :=gx$.  The
differential $d\alfa_e$ is the map $\xi \mapsto \xi_M (x)$. The following result is well-know. A proof is given in \cite{biliotti-ghigi-heinzner-2}.
\begin{prop}\label{linearization1}
  \label{tangent}
  If $\beta \in \liep$ and $x \in \Crit(\mu^\beta_\liep)$ then
  \begin{gather*}
    D^2\mup^\beta(x) = d \beta_M (x).
  \end{gather*}
  Moreover $d\alfa_e (\mathfrak r^{\beta\pm} ) \subset V_\pm$ and $d\alfa_e(
  \lieg^\beta) \subset V_0$.  If $M$ is $G$-homogeneous these are
  equalities.
\end{prop}
\begin{cor}
  \label{MorseBott}
  For every $\beta \in \liep$, $\mupb$ is a Morse-Bott function.
\end{cor}
\begin{proof}
Corollary \ref{slice-cor-2} implies that $\Crit(\mu^\beta_\liep)$ is a smooth
    submanifold. Since $T_x \Crit(\mu^\beta_\liep) = V_0$ for $x\in \Crit(\mu^\beta_\liep)$, the
  first statement of Proposition \ref{tangent} shows that the Hessian
  is nondegenerate in the normal directions.
\end{proof}
\begin{cor}\label{critical}
If $M$ is $G$-homogenous then $G^{\beta}$-orbits are open and closed in $\mathrm{Crit}\, \mup^\beta$.
\end{cor}
\begin{proof}
Since $T_x \mathrm{Crit}\, \mup^\beta = T_x G^{\beta} \cdot x$ for $x\in  \mathrm{Crit}\, \mup^\beta$, the result follows.
\end{proof}
Let $c_1 > \cds > c_r$ be the critical
  values of $\mup^\beta$. The corresponding level sets of $\mup^\beta$, $C_i:=(\mup^\beta)\meno ( c_i)$ are submanifolds which are union of
  components of $\Crit(\mup^\beta)$. The function $\mup^\beta$ defines a gradient flow generated by its gradient which is given by $\beta_M$.
By Theorem \ref{line}, it follows that for any $x\in M$ the limit:
  \begin{gather*}
    \phi_\infty (x) : = \lim_{t\to +\infty }
    \exp(t\beta ) x,
  \end{gather*}
  exists.
Let us denote by $W_i^\beta$ the \emph{unstable manifold} of the critical component $C_i$
  for the gradient flow of $\mup^\beta$:
  \begin{gather}
    \label{def-wu}
    W_i^\beta := \{ x\in M: \phi_\infty (x) \in C_i \}.
  \end{gather}
Applying Theorem \ref{line}, we have the following well-known decomposition of $M$ into unstable manifolds with respect to $\mup^\beta$.
\begin{teo}\label{decomposition}
In the above assumption, we have
\begin{gather}
    \label{scompstabile}
    M = \bigsqcup_{i=1}^r W_i^\beta,
  \end{gather}
and for any $i$ the map:
  \begin{gather*}
    (\phi_\infty)\restr{W_i} : W_i^\beta \ra C_i,
  \end{gather*}
  is a smooth fibration with fibres diffeomorphic to $\R^{l_i}$ where
  ${l_i}$ is the index (of negativity) of the critical submanifold
  $C_i$
\end{teo}
Let $\beta \in \liep$ and let $\mathrm{Max}_M (\beta):=\{x\in M:\, \mup^\beta(x)=\mathrm{max}_{y\in M} \mup^\beta\}$.
By Proposition \ref{linearization1}, $\mathrm{Max}_M (\beta)$
is a smooth, possibly disconnected, locally closed submanifold of $M$. The following result is proved in \cite[Lemma 30 and Proposition 31]{bsg}.
\begin{prop}\label{parabolic-preserve-maximun}
$\mathrm{Max}_M (\beta)$ is $G^{\beta+}$-invariant. Moreover, $R^{\beta+}$ acts trivially on $\mathrm{Max}_M (\beta)$ and the $G^{\beta+}$-action
on $\mathrm{Max}_M (\beta)$ admits a compact orbit which is also a $K^\beta$-orbit.
\end{prop}
Using an $\mathrm{Ad}(K)$-invariant  inner product of  $\liep$, we define $\nu_\liep (z):=\frac{1}{2}\parallel \mup(z)\parallel^2$.
The function $\nu_\liep$ is $K$-invariant and it is called \emph{the norm square function}.
In \cite{heinzner-schwarz-stoetzel} (see Corollary 6.11 and Corollary 6.12 p. $21$) the following result is proved.
\begin{prop}\label{heinzner-maximun}
Let $x\in M$. Then:
\begin{itemize}
\item if $\nu_\liep$ restricted to $G\cdot x$ has a local maximum at $x$, then $G\cdot x=K\cdot x$
\item if $G\cdot x$ is compact, then $G\cdot x =K\cdot x$
\end{itemize}
\end{prop}
A strategy to analyzing the $G$-action on $M$ is to view $\nu_\liep$ as generalized Morse function.
In \cite{heinzner-schwarz-stoetzel} the authors proved the existence of a smooth $G$-invariant stratification of $M$ and they studied its properties. If $\beta \in \liep$ is a critical value, then we may associate a stratum $S_\beta$ which is a $G$-invariant locally closed submanifold of $M$. The stratification Theorem proves that $S_\beta$ only depends upon $K\cdot \beta$ which is called critical orbit. Indeed, $S_\beta =S_{\beta'}$ if and only if $K\cdot \beta =K\cdot \beta'$ and
\[
M=\bigsqcup_{\beta } S_\beta,
\]
where $\beta$ runs through a complete set of representative $K$-orbits in the set of critical value.

Let $\lia \subset \liep$ be an Abelian subalgebra. The $A=\exp(\lia)$-gradient map is given by  $\pi_\lia \circ \mup$, where  $\pi_\lia:\liep \lra \lia$
denotes the orthogonal projection of $\liep$ onto $\lia$. If $M$ is connected and compact, applying the convexity Theorem along $A$-orbits
\cite{atiyah-commuting,heinzner-schutzdeller},  then $\mua(M^A)$ is a finite set and  $\mua(M)\subseteq \mathrm{conv} (\mua(M^A))$,
where  $M^A=\{p\in M:\, A\cdot p=p\}$ \cite[Proposition 3.1]{bgh-israel-p} . In particular $\mathrm{conv} (\mua(M))=\mathrm{conv}(\mua(M^A))$. Therefore,
if $\mua(M)$ is convex then $\mua(M)$ is the convex hull of $\mua(M^A)$ and so a polytope \cite{schneider-convex-bodies}.

Let $\mathcal E$ denote the convex hull of $\mup(M)$ and let $\lia\subset \liep$ be a maximal Abelian subalgebra.
Since $\mathcal E\cap \lia=\pi_\lia (E)=\mathrm{conv} ( \mua(M))$, applying Theorem \ref{convex-reduction} we get $\mathcal E=K\mathrm{conv}(\mua(M))$.
\section{Projective representations of real reductive Lie groups}\label{section}
Let $G$ be a connected real semisimple noncompact Lie group and let $\rho:G \lra \mathrm{SL}(V)$ be an irreducible
representation on a finite dimensional real vector space $V$. We identify $G$ with $\rho(G)\subset \mathrm{GL}(V)$ and we assume that
$G$ is closed and there exists $K$-invariant scalar product $\mathtt g$ on $V$ such that $G=K\exp(\liep)$, where
$K\subset \mathrm{SO}(V,\mathtt g)$, $\liep= \mathrm{Sym}_0 (V,\mathtt g)\cap \lieg$ and  $\lieg$ denotes the Lie algebra of $G$. Here,  $\mathrm{Sym}_o (V,\mathtt g)$ denotes the set of
symmetric endomorphisms with trace zero.
Roughly speaking, if we identify $V$ with $\R^n$, then $G$ is  a closed and compatible subgroup of $\mathrm{SL}(n,\R)$.
Hence $G=K\exp(\liep)$, where $K:=G\cap \mathrm{SO}(n)$  and $\liep\subset \mathrm{Sym}_o (n)$. In particular, $K$ is a maximal compact subgroup of $G$ and
$\lieg=\mathfrak k \oplus \liep$ is the Cartan decomposition of $G$, where $\mathfrak k$ is the Lie algebra of $K$, induced by the
Cartan decomposition of $\mathrm{SL}(n,\R)$. Now,   $G\subset \mathrm{SL}(n,\C)$ is compatible as well and $\liek \cap i\liep =\{0\}$. By \cite[Proposition 2 p. $4$]{heinz-stoezel},
the Zariski closure of $G$ in $\mathrm{SL}(n,\C)$ is given by  $U^\C$, where $U$ is the compact connected semisimple Lie group with Lie algebra
 $\mathfrak u=\mathfrak k \oplus i\mathfrak  p$.  In particular $G$ is a compatible real form of $U^\C$. In the sequel we always assume that both the $G$-action on $\R^n$ and the $G$-action on $\C^n$ are irreducible. Hence the
$U^\C$-action on $\C^n$ is irreducible as well. We are interested to the natural projective representation of $G$ on
$\mathbb P (\R^n)$. The $G$-action on $\mathbb P (\R^n)$ admits a $G$-gradient map.

Let $B:\mathfrak{sl}(n,\C) \times \mathfrak{sl}(n,\C) \lra \R$ be the symmetric bilinear form given by
\[
B(X,Y)=\mathrm{Re}(\mathrm{Tr}(XY)).
\]
It is an $\mathrm{Ad}(\mathrm{SL}(n,\C))$-invariant and nondegenerate bilinear form. Indeed, the splitting $\mathfrak{sl}(n,\C) =\mathfrak{su}(n) \oplus i \mathfrak{su}(n)$ is $B$-orthogonal,  $B$ is negative definite on $\mathfrak{su}(n)$ and positive definite on $i\mathfrak{su}(n)$.
We define $\langle \cdot , \cdot \rangle$ on $\mathfrak{sl}(n,\C)$ as follows: $\scalo=-B(\cdot , \cdot)$ on $\mathfrak{su}(n)$; $\scalo =B(\cdot,\cdot)$ on $i\mathfrak{su}(n)$; $\langle \mathfrak{su}(n),i \mathfrak{su}(n)\rangle =0$. $\scalo$ is $\mathrm{Ad}(\mathrm{SU}(n))$-invariant scalar product on $\mathfrak{sl}(n,\C)$ such that the multiplication by $i$ defines an isometry of $\mathfrak{su}(n)$ onto $i\mathfrak{su}(n)$.
The $U$-action on $\mathbb P (\C^n)$ is Hamiltonian with momentum map
\[
\mu:\mathbb P(\C^n) \lra \liu, \qquad \mu=\pi_\liu \circ \Phi,
\]
where $\Phi(z)=-\frac{i}{2} \left(\frac{zz^*}{\parallel z \parallel^2} -\frac{1}{n}\mathrm{Id}_n\right)$ is the momentum map of the $\mathrm{SU}(n)$-action on $\mathbb P(\C^n)$ and $\pi_\liu$ is the orthogonal projection of $\mathrm{su}(u)$ onto  $\liu$ \cite{kirwan}. The momentum map satisfies the following conditions:
\begin{itemize}
\item for any $z\in \mathbb P(\C^n)$ and any $g\in U$, we have $\mu(gz)=\mathrm{Ad}(g)(\mu(z))$;
\item  If $\xi \in \liu$, we denote the induced vector field by $\xi_{\mathbb P(\C^n)} (p)=\desudtzero \exp(t\xi) p$. Let $\mu^\xi \in C^\infty (Z)$ be
the function $\mu^\xi(z) := \langle \mu(z),\xi\rangle$. Then $d\mu^\xi =
i_{\xi_{\mathbb P(\C^n)}} \omega$.
\end{itemize}
Let $z \in \mathbb P(\C^n)$.  Then $\mup (z) \in \liep$ denotes the component of $i\mu(z)$ in the direction of $\liep$.  In other words we require
that $\langle \mup (z) , \beta \rangle = \langle i \mu(z) , \beta\rangle=\langle \mu(z), -i \beta \rangle$ for any
$\beta \in \liep$.  We have thus defined the $G$-\emph{gradient map}
\begin{gather*}
  \mu_\liep : \mathbb P(\C^n) \lra \liep,
\end{gather*}
which satisfies the following conditions:
\begin{enumerate}
\item $\mup(kz)=\mathrm{Ad}(k)(\mup(z))$, for any $k\in K$ and for any $z\in \mathbb P(\C^n)$;
\item for any $\beta \in \liep$, let $\mup^\beta \in C^\infty (\mathbb P(\C^n))$ be the function $\mup^\beta(z)=\langle \mup(z),\beta \rangle=\mu^{-i\beta} (z)$. Then the gradient of $\mup^\beta$, with respect to the Riemannian metric induced by the K\"ahler structure, is given by $\beta_{\mathbb P(\C^n)}$.
\end{enumerate}
If $X\subset \mathbb P(\C^n)$ is a connected $G$-stable real submanifold of $\mathbb P(\C^n)$, we consider $\mup$ as a $K$-equivariant mapping $\mup:X \lra \liep$ such that for any $\beta \in \liep$, the gradient of the smooth function $\mup^\beta$ is given by $\beta_X$, where the gradient is computed with respect to the induced Riemannian metric on $X$.
\begin{lemma}\label{restrizione}
For any $z\in \mathbb P(\R^n)$ we have
$
\mup(z)=i\mu (z)
$
\end{lemma}
\begin{proof}
Let $z\in\mathbb P(\R^n)$. Then $\Phi(z)=-\frac{i}{2} \left(\frac{zz^T}{\parallel z \parallel^2} -\frac{1}{n}\mathrm{Id}_n\right) \in i\mathrm{Sym}_0 (n)$. Since $\langle \mathfrak{so}(n),i\mathrm{Sym}_0 (n) \rangle =0$, it follows that
 $\mu(z)$ is the orthogonal projection of $\Phi$ onto $i\liep$ and so the result follows.
\end{proof}
\begin{prop}\label{closed-orbit}
There exists $v\in \mathbb P(\R^n)$ such that $U^\C \cdot v$ is the unique compact orbit of the $U^\C$-action on $\mathbb P(\C^n)$. Moreover,
$G\cdot v$ is the unique compact orbit of the $G$-action on $U^\C \cdot v$. It is also a $K$-orbit and a Lagrangian submanifold of $U^\C \cdot v$.
\end{prop}
\begin{proof}
Let $\lia\subset \liep$ be a maximal Abelian subalgebra. The centralizer of $\lia$ is $\lieg$ is compatible and is given
by $\mathfrak{z}_{\lieg} (\lia)=\mathfrak{z}_\liek (\lia) \oplus \lia$.

Let $\mathfrak b$ be a maximal Abelian subalgebra of
$ \mathfrak{z}_\liek (\lia)$ and
let $\mathfrak h=\mathfrak b \oplus \mathfrak a$. Then $\mathfrak h^\C=(\mathfrak b \oplus i\mathfrak a)\oplus i (\mathfrak b \oplus i\mathfrak a)$
is a maximal Abelian subalgebra of $\liu^\C$. Since the $\mathfrak a$-action on $\C^n$ is the complexification of the $\mathfrak a$-action on $\R^n$,
it follows that the eigenspaces of the $\mathfrak a$-action on $\C^n$ are the complexification of the eigenspaces of the $\mathfrak a$-action on $\R^n$.
Since $U^\C$ acts irreducible on $\C^n$, the eigenspaces corresponding to the weight space is one dimensional. If $z$ is a nonzero vector belonging to the
weight space, then it is a an eigenvector for the $\lia$-action on $\C^n$. This implies $\pi(z) \in \mathbb P (\R^n)$, where
$\pi:\C^n \setminus \{0\} \lra \mathbb P (\C^n)$ is the natural projection. We denote $v=\pi(z)$.
By Borel-Weyl Theorem $U^\C \cdot v$ is the unique compact orbit of the $U^\C$-action on $\mathbb P(\C^n)$
which coincides with the unique complex orbit of $U$ \cite{huckleberry-introduction-DMV}. Since $\mathbb P (\R^n)$ is Lagrangian and $G$ is a real form of $U^\C$,
by \cite[Proposition 9]{bsb} it follows that $G\cdot v$ is compact as well and Lagrangian.
By Proposition \ref{heinzner-maximun}, $G\cdot v$ is a $K$-orbit.
Finally, applying a  Theorem of Wolf \cite{wolf}, see also \cite{heinzner-stoetzel-global},
it follows that $G\cdot v$ is the unique compact $G$-orbit in $U^\C \cdot v$.
\end{proof}
In the sequel we always denote by $\OO'$ the unique compact orbit of the $U^\C$-action on $\mathbb P(\C^n)$ and by $\OO$ the unique compact orbit of the $G$-action on $\OO'$.
\begin{cor}\label{intersection}
$\OO' \cap \mathbb P(\R^n)=\OO$.
\end{cor}
\begin{proof}
By Proposition \ref{heinzner-maximun}, $\OO'$ is a $U$-orbit and the set where the norm square momentum map achieves its maximum.
Now, if $z\in \mathbb P(\C^n)$, then
\begin{equation}\label{splitting}
\parallel \mu(z) \parallel^2 = \parallel \mu_{\mathfrak k} (z) \parallel^2 + \parallel \mup (z) \parallel^2,
\end{equation}
where $\mu_{\mathfrak k}$ is the momentum map of the $K$-action on $\mathbb P(\C^n)$.
Since $\OO'=U\cdot v$ and $v\in \mathbb P(\R^n)$, keeping in mind Lemma \ref{restrizione},  $v$ achieves the maximum of the norm square gradient map as well.  Moreover, if $z\in U \cdot v \cap \mathbb P(\R^n)$ then
\[
\parallel \mu(z) \parallel =\parallel \mu (v ) \parallel =\parallel \mup (z) \parallel,
\]
and so  $z$ achieves the maximum of the norm square gradient map. By Proposition \ref{heinzner-maximun}, $G\cdot z$ is compact and it is contained in $\OO'$. By Proposition \ref{closed-orbit}, $G$ has a unique compact orbit on $\OO'$ concluding the proof.
\end{proof}
The map
\[
T:\mathbb P(\C^n) \lra \mathbb P(\C^n), \qquad [x] \mapsto [\overline x ],
\]
is an anti-holomorphic isometric involution. The Lie algebra $\liu=\mathfrak k \oplus i \liep$ is invariant with respect to the matrix conjugation induced by $T$. Hence, keeping in mind that the exponential map of $U$ is surjective \cite{knapp}, it follows $\overline{U}=U$. Therefore, keeping in mind that $v\in \mathbb P(\R^n)$, we get
\[
\OO'=U^\C \cdot v=U\cdot v=\overline{U\cdot v}.
\]
This implies that complex conjugation $T$ induces an anti-holomorphic isometry on $\OO'$ whose fixed point set is given by
\[
\OO' \cap \mathbb P(\R^n)=\OO.
\]
Since the fixed point set of an isometry is a totally geodesic submanifold \cite{bm}, we get the following result.
\begin{teo}\label{involution}
$\OO$ is the fixed point set of an anti-holomorphic involutive isometry of $\OO'$. In particular $\OO$ is a totally geodesic submanifold of $\OO'$.
\end{teo}
We claim that the vice-versa holds.

Let $\rho:U^\C \lra \mathrm{SL}(n,\C)$ be an holomorphic irreducible representation of a semisimple complex Lie group. Let $G$ be a noncompact real form of $U^\C$.
Assume there exists an anti-holomorphic involution $T$ of $\mathbb P(\C^n)$ preserving $\OO'$ and such that $\OO$ is contained in the fixed point set of $T$.  Now, the application $T$
is induced by an anti-linear map $T:\C^n \lra \C^n$ such that $T^2=\mathrm{Id}_n$.  Let $V=\mathrm{Ker}(T-\mathrm{Id}_n)$. Since $T\circ J=-J\circ T$, it follows that $\C^n=V\oplus J(V)$. and so the fixed point set of $T$ restricted on $\OO'$ is given by $\OO' \cap \mathbb P(V)$. This implies that $\OO\subseteq \mathbb P(V)$. Hence $G$ preserves a real subspace $W$ of $V$ and so $W^\C$ is preserved by $U^\C$. This implies that $W=V$ and $V^\C=\C^n$. Summing up, we have proved the following result.
\begin{teo}
In the above assumption, there exists a real subspace $V\subset \C^n$ such that $G$ acts irreducibly on $V$ and $V^\C=\C^n$.
\end{teo}
Let $\beta \in \liep$ and let
$
\mup^\beta:\mathbb P (\R^n) \lra \R.$
Let $\lambda_1>\dots> \lambda_k$ be the eigenvalues of $\beta$.  We denote by $V_1,\dots, V_k$ the corresponding eigenspaces. In view of the orthogonal decomposition $\R^n=V_1\oplus\dots\oplus V_{k}$, $\mup^\beta$ is given by
\[
\mup^\beta ([x_1+\cdots+x_k])=\frac{\lambda_1 \parallel x_1 \parallel^2 + \cdots +\lambda_k \parallel x_k \parallel^2}{\parallel x_1 \parallel^2 + \cdots +\parallel x_k \parallel^2}.
\]
Therefore $\mathrm{Max}_{\mathbb P (\R^n)} (\beta)=\mathbb P(V_1)$. The gradient flow of $\mup^\beta$ is given by
\[
\R \times \mathbb P(\R^n) \lra \mathbb P(\R^n), \qquad \big(t,[x_1+\cdots + x_k] \big) \mapsto [e^{t\lambda_1}x_1+\cdots+e^{t\lambda_1}x_1].
\]
Hence the critical points of $\mup^\beta$ are given by $\mathbb {P}(V_1)\cup\dots\cup \mathbb {P}(V_{k})$ and the corresponding unstable manifolds are given by:
$$\label{unstable}
W_1^\beta=\mathbb {P} (\R^n)\setminus \mathbb {P}(V_2\oplus\dots\oplus V_k),
$$
$$
W_2^\beta=\mathbb {P}(V_2\oplus\dots\oplus V_k)\setminus \mathbb {P}(V_3\oplus\dots\oplus V_k),
$$
$$
\vdots
$$
$$
W_{k-1}^\beta=\mathbb {P}(V_{k-1}\oplus V_k)\setminus \mathbb {P}(V_k),
$$
$$
W_k^\beta=\mathbb {P}(V_k).$$
\begin{lemma}\label{rm}
Let $M$ be a closed $G$-stable submanifold of $\mathbb P(\R^n)$ and let  $\beta \in \liep$.  Then
\[
\mathrm{max}_{x\in M} \mup^\beta=\mathrm{max}_{x\in \mathbb P(\R^n)} \mup^\beta=\mathrm{max}_{x\in \mathbb P(\C^n)} \mup^\beta=\mathrm{max}_{x\in \OO'} \mup^\beta.
\]
\end{lemma}
\begin{proof}
If $W_1^\beta \cap M=\emptyset$, then $M\subset \mathbb P(V_2 \oplus \cdots \oplus V_k)$ and so there  exists a proper $G$-stable subspace of $\R^n$. A contradiction.

Let $x\in W_1^\beta \cap M$. Then
\[
\lim_{t\mapsto +\infty} \exp(t\beta)x \in M \cap V_1,
\]
proving $\mathrm{max}_{x\in M} \mup^\beta=\mathrm{max}_{x\in \mathbb P(\R^n)} \mup^\beta$.

Since $\beta$ is a real matrix, we get the following orthogonal splitting $\C^n=V_1^\C\oplus\dots\oplus V_{k}^\C$, with respect to the canonical Hermitian scalar product, of eigespaces of $\beta$. In particular, the function $\mup^\beta:\mathbb P(\C^n) \lra \R$, is given by
\[
\mup^\beta ([x_1+\cdots+x_k])=\frac{\lambda_1 \parallel x_1 \parallel^2 + \cdots +\lambda_k \parallel x_k \parallel^2}{\parallel x_1 \parallel^2 + \cdots +\parallel x_k \parallel^2}
\]
and so $\mathbb {P}(V_1^\C)\cup\dots\cup \mathbb {P}(V_{k}^\C)$ are the critical points of $\mup^\beta$. The corresponding unstable manifolds are given by:
$$
W_1^\beta=\mathbb {P} (\C^n)\setminus \mathbb {P}(V_2^\C\oplus\dots\oplus V_k^\C),
$$
$$
W_2^\beta=\mathbb {P}(V_2^\C\oplus\dots\oplus V_k^\C)\setminus \mathbb {P}(V_3^\C\oplus\dots\oplus V_k^\C),
$$
$$
\vdots
$$
$$
W_{k-1}^\beta=\mathbb {P}(V_{k-1}^\C \oplus V_k^\C )\setminus \mathbb {P}(V_k^\C),
$$
$$
W_k^\beta=\mathbb {P}(V_k^\C).$$
This implies $\mathrm{max}_{x\in \mathbb P(\C^n)} \mup^\beta=\lambda_1$.

Let $\OO'$ be the unique compact $U^\C$-orbit in $\mathbb P(\C^n)$. Since $G$ acts irreducibly on $\C^n$, it follows that the unstable manifold $W_1^\beta$ of the Morse-Bott function $\mup^\beta: \mathbb P(\C^n) \lra \R$ intersects $\OO'$. As before, one has
\[
\maxm{\beta}{\mathcal O'}=\maxm{\beta}{\mathbb P (\C^n)} \cap \OO' ,
\]
concluding the proof.
\end{proof}
\begin{prop}\label{image}
Let $\lia \subset \liep$ be an Abelian subalgebra. Then
\[
\mua(\mathcal O)=\mua(\mathbb P(\R^n))=\mua(\mathbb P(\C^n))=\mua(\mathcal O').
\]
\end{prop}
\begin{proof}
Since $\OO$ is a $K$-orbit, keeping in mind that $\mup$ is $K$-equivariant and $\mua=\pi_\lia \circ \mup$, by a Theorem of Kostant \cite{kostant-convexity} it follows that $\mua(\mathcal O)$ is a polytope. $\mua(\mathbb P (\R^n))$ is a polytope as well \cite{heinzner-schutzdeller} and by Lemma \ref{rm} we get
\[
\maxm{\beta}{\mathcal O}=\maxm{\beta}{\mathbb P (\R^n)} \cap \OO,
\]
for any $\beta \in \liep$. Applying Proposition \ref{convex-criterium}, we have $\mua(\mathcal O)=\mua(\mathbb P (\R^n))$.

Lemma \ref{rm} also proves
\[
\maxm{\beta}{\mathbb P (\C^n)}=\maxm{\beta}{\mathbb P (\R^n)}=\maxm{\beta}{\OO'},
\]
for any $\beta \in \liep$. Applying again Proposition \ref{convex-criterium}, we get $\mua(\mathbb P (\C^n) )=\mua(\mathbb P (\R^n))=\mua(\mathcal O')$, concluding the proof.
\end{proof}
Let $v\in \mathbb P(\R^n)$ be such that $G\cdot v=\OO$ and $\OO'=U^\C \cdot v$.
\begin{prop}\label{gradient-closed-orbit}
The restricted $G$-gradient map
\[
\mup:K\cdot v \lra K\cdot \mup(v),
\]
is a diffeomorphism.
\end{prop}
\begin{proof}
$U^\C \cdot v$ is compact, and $U^\C\cdot v=U\cdot v$. Then $U_{v}=U^{\mu(v)}$ and so the centralizer of $\mu(v)$ \cite{guillemin}. By Lemma \ref{restrizione} we have $\mu(v)=i\mu (v)$. Therefore, keeping in mind that $K_v=U_v \cap K$, we have $K_v=K^{\mup(v)}$. Since $\mup$ is $K$-equivariant the result is proved.
\end{proof}
Let $\beta \in \liep$ and let $\lambda_1>\dots> \lambda_k$ be the eigenvalues of $\beta$ acting on $\R^n$.  We denote by $W$ the eigenspace corresponding to $\lambda_1$.
\begin{teo}\label{parabolici}
The following results hold true:
\begin{enumerate}
\item $W$ is the unique subspace of $\R^n$ such that $G^{\beta+}$ acts irreducibly on it;
\item $\maxm{\beta}{\mathcal O}$ is connected and it coincides with the unique compact orbit of the $G^{\beta+}$-action on $\mathcal O'$. $\maxm{\beta}{\mathcal O}$  completely characterizes $W$;
\item $\maxm{\beta}{\mathcal O}$ is contained in the unique compact orbit of $(U^\C)^{\beta+}$-action on $\mathbb P(\C^n)$ given by $\maxm{\beta}{\mathcal O'}$. Moreover, $\maxm{\beta}{\mathcal O}$  is a Lagrangian submanifold  and the fixed point set of an anti-holomorphic involutive isometry of $\maxm{\beta}{\mathcal O'}$, and so it is totally geodesic.
\end{enumerate}
\end{teo}
\begin{proof}
Let $\mathcal O=K\cdot v$. By  \cite[Theorem 1.2]{biliotti-ghigi-heinzner-2} it follows that $\maxm{\beta}{K\cdot \mup(v)}$ it is $(K^\beta)^o$-orbit and so it is connected. By Proposition \ref{gradient-closed-orbit}, keeping in mind that $\mup$ is $K$-equivariant, we have $\maxm{\beta}{\mathcal O}$ is connected. By Proposition \ref{parabolic-preserve-maximun}, $R^{\beta+}$ acts trivially on $\maxm{\beta}{\mathcal O}$. Applying \cite[Theorem 1]{bj}, keeping also in mind Proposition \ref{parabolic-preserve-maximun}, it follows that $G^{\beta+}$ has a unique compact orbit in $\mathcal O$, which is a $(K^{\beta})^o$-orbit, and this orbit coincides with  $\maxm{\beta}{\mathcal O}$. By Lemma \ref{rm} $\maxm{\beta}{\mathcal O}\subset \maxm{\beta}{\mathcal O'}$ which is the  unique compact orbit of the $(U^\C)^{\beta+}$-action on $\mathbb P(\C^n )$ \cite{bj}. Hence the first part of $(b)$ and $(c)$ are proved.

By Proposition \ref{parabolic-preserve-maximun},  $G^{\beta+}$ preserves $\mathbb P(W)$ and $R^{\beta+}$ acts trivially on $\mathbb P(W)$. Since $G^{\beta}$ is reductive, $W$ splits as $G^\beta$-invariant, and so as $G^{\beta+}$-invariant, subspaces \cite{knapp}.

Let  $L\subseteq W$ be a $G^{\beta+}$-invariant subspace. Since $\mathbb P(W^\C)=\mathrm{Max}_{\mathbb P(\C^n)}(\beta)$, it follows that $W^\C$ is the unique subspace of $\C^n$ such that $(U^\C)^{\beta+}$ acts irreducibly on it \cite[Theorem 2]{bsg}. By Proposition \ref{parabolic-preserve-maximun}  the unipotent part of $(U^\C)^{\beta+}$ acts trivially on $\mathbb P(W^\C)$. Hence,  keeping in mind that $(\mathfrak g^\beta)^\C=(\mathfrak u^\C)^\beta$, it follows that $L^\C$ is a $(U^\C)^{\beta+}$-invariant subspace as well. This implies $L^\C=W^\C$ and so $L=W$. This proves $(a)$ and $(b)$.

The $(U^{\C})^{\beta}$-action on $W^\C$ is irreducible and the unipotent part of the parabolic subgroup $(U^{\C})^{\beta+}$ acts trivially on $\mathbb P(W^\C)$. By the Schur Lemma \cite{knapp} the center of $(U^\C)^\beta$ acts  trivially on $\mathbb P(W^\C )$. This implies that the center of $G^{\beta}$ acts trivially on $\mathbb P(W^\C )$, and so on $\mathbb P(W)$.

Let $G^\beta_{ss}$ be the connected subgroup of $G^\beta$ whose Lie algebra is $[\mathfrak g^\beta, \mathfrak g^\beta]$. Let  $(U^\C)^\beta_{ss}$ be the connected subgroup of $(U^\C)^\beta$ whose Lie algebra is $[(\mathfrak u^\C)^\beta, (\mathfrak u^\C)^\beta]$. $(G^\beta)_{ss}$ is semisimple, closed \cite{knapp} and a real form of the closed complex semisimple Lie group $(U^\C)^\beta_{ss}$.  By the above discussion both the $G^\beta_{ss}$-action on $W$ and the $(U^\C)^\beta_{ss}$-action on $W^\C$ are irreducible. The unique compact orbit of the $G^{\beta+}$-action on $\OO$ is a compact $G^\beta_{ss}$-orbit. The unique compact orbit of the $(U^{\C})^{\beta+}$-action on $\mathbb P(\C^n )$ is a compact $(U^\C)^\beta_{ss}$-orbit.  By Proposition \ref{closed-orbit} and Theorem \ref{involution},   $\maxm{\beta}{\mathcal O}$ is a Lagrangian submanifold and the fixed point set of an anti-holomorphic involutive isometry of $\maxm{\beta}{\OO'}$, concluding the proof.
\end{proof}
\begin{cor}
Let $\beta \in \liep$ and let $W$ be the eigenspace corresponding to the maximum eigenvalue of $\beta$ acting on $\R^n$. Let $z\in \mathbb P(\C^n)$ be such that $(U^{\C})^{\beta+} \cdot z$ is compact. Then  $(U^{\C})^{\beta+}\cdot z \cap \mathbb P(W)$ is the unique compact orbit of the $G^{\beta+}$-action on $\OO$.
\end{cor}
\begin{proof}
 By the proof of the above Theorem, the unique compact $G^{\beta+}$-orbit on $\mathcal O$ is a compact $G^\beta_{ss}$-orbit contained in the unique compact orbit of the $(U^\C)^\beta_{ss}$-action on $\mathbb P(W^\C)$. The result follows by Corollary \ref{intersection}.
\end{proof}
\begin{cor}\label{oi}
Let $Q$ be a parabolic subgroup of $G$ and let $L\subset \C^n$ be a $Q$-invariant complex subspace. If the $Q$-action on $L$ is irreducible, then  there exists an invariant $Q$-subspace $W$ of $\R^n$ such that the $Q$-action on $W$ is irreducible and $L=W^\C$.
\end{cor}
\begin{proof}
Let $\mathfrak q$ be the Lie algebra of $Q$. Then $\mathfrak q^\C$ is a parabolic subalgebra of $\mathfrak g^\C$ \cite{borel-ji-libro}. Let $\tilde Q$ denote the parabolic subgroup of $U^\C$ whose Lie algebra is $\mathfrak q^\C$. The $\tilde Q$-action on $\mathbb P (\C^n)$ has a unique compact orbit and it is contained in $\mathcal O'$. Since the unipotent part of $\tilde Q$ normalizes $\tilde Q$ and the $\tilde Q$ acts irreducibly on $L$, by Engel Theorem the unipotent part acts trivially on $\mathbb P (L)$. Hence $\tilde Q$ has a unique compact orbit contained in $\mathbb P(L)$ that we denote by $\OO(\tilde Q)$.

Let $\beta \in \liep$ be such that $Q=G^{\beta+}$. By the above discussion $R^{\beta+}$ acts trivially on $\mathbb P(L)$. By Proposition \ref{heinzner-maximun}, $(G^\beta)^o$ has a compact orbit on $\OO(\tilde Q)$ which is a $(K^\beta)^o$-orbit. Since $G^\beta$ has a finite number of connected components and any connected component intersect $K^\beta$, it follows that $G^\beta$ has a compact orbit on $\OO(\tilde Q)$.
By Theorem \ref{parabolici}, it is the unique compact $G^{\beta+}$-orbit  in $\OO$. This orbit is contained in $\mathbb P(\R^n)$. Hence its real span coincides with $\mathbb P(W)$, where $W$ is the unique $G^{\beta+}$-invariant subspace of $\R^n$ such that $G^{\beta+}$ acts irreducibly on it. Therefore $W^\C$ is contained in $L^\C$ and it is $\tilde Q$-invariant. This implies $L=W^\C$.
\end{proof}
Let $v\in \mathbb P(\R^n)$ be such that $U^\C \cdot v=\OO'$ and $G\cdot v=\OO$.
Let $\beta \in \liep$. By Theorem \ref{parabolici}, $\maxm{\beta}{\OO}$ is the unique compact $G^{\beta+}$-orbit contained in $\OO$. Now,
\[
\mup(\maxm{\beta}{\OO})=\maxm{\beta}{K\cdot \mup(v_\rho)}.
\]
By \cite[Corollary 3.1 p. $593$ and Proposition 3.9 p. $599$]{biliotti-ghigi-heinzner-2}, $\maxm{\beta}{K\cdot \mup(v_\rho)}$ is the unique compact orbit of the $G^{\beta+}$-action on $K\cdot \mup(v)$.  By Proposition \ref{gradient-closed-orbit} and Theorem \ref{parabolici}, we get the following result
\begin{prop}\label{gi}
Let $z\in \OO$. Then $G^{\beta+} \cdot z$ is compact if and only if $G^{\beta+}\cdot \mup(z)$ is compact and
\[
\mup(G^{\beta+}\cdot  z)=G^{\beta+}\cdot \mup(z)
\]
\end{prop}
We conclude this section showing how the gradient flow of the Morse Bott function $\mup^\beta$ determines the unique compact orbit of the $G^{\beta+}$-action on $\OO$.

Let $W\subset \R^n$ be a real subspace. We denote by
\[
\pi_W : \R^n \lra W,
\]
the orthogonal projection and by
\[
\hat\pi_W : \mathbb P(V) \setminus{\mathbb P(W^\perp)}  \lra \mathbb P(V), \qquad \hat\pi_W ([v])=[\pi_W (v)],
\]
its projectivization. Let $\beta \in \liep$ and let $W\subset \R^n$ be the eigenspace corresponding to the maximum eigenvalue of $\beta$.
Since $\mathbb P(W)=\maxm{\beta}{\mathbb P(\R^n)}$, the domain of the map $\hat\pi_{W}$ is the unstable manifold of the maximum of $\mup^{\beta}$. Moreover,  $\hat\pi_W$ coincides with the gradient flow $\phi_\infty$. Indeed, let $\lambda_1>\dots> \lambda_k$ be the eigenvalues of $\beta$. Let $V_2,\ldots,V_k$ be the eigenspaces associated to $\lambda_2,\ldots,\lambda_k$. Then
\[
\begin{split}
\lim_{t\mapsto +\infty} \exp(t\beta) [x_1+x_2+\cdots+x_k]&=\lim_{t\mapsto +\infty} [e^{t\lambda_1}x_1+e^{t\lambda_2}x_2+\cdots +e^{t\lambda_k}x_k]
\\ &=\lim_{t\mapsto +\infty} [x_1+e^{t(\lambda_2-\lambda_1)}x_2+\cdots +e^{t(\lambda_k-\lambda_1)}x_k]\\ &=[x_1] \\ &=\hat\pi_W (x).
\end{split}
\]
By Theorem \ref{decomposition}, an unstable manifold flows into the corresponding critical set. Hence $\hat\pi_{W} (\mathcal O )=\maxm{\beta}{\mathcal O}$.

Let $W^\C\subset \C^n$. The $(U^\C)^{\beta+}$-action on $W^\C$ is irreducible and
\[
\mathbb P(W^\C)=\maxm{\beta}{\mathbb P(\C^n)}.
\]
The prjojectivazion of the orthogonal projection onto $W^\C$ with respect to the canonical Hermitian product, i.e.
\[
\hat\pi_{W^\C} : \mathbb P(\C^n) \setminus{\mathbb P((W^\C)^\perp)}  \lra \mathbb P(V), \qquad \hat\pi_{W^\C} ([v])=[\pi_{W^\C} (v)],
\]
is the gradient flow of $\mup^\beta$ restricted to the unstable manifold corresponding to the maximum. Therefore $\hat\pi_{W^\C}(\mathcal O')=\maxm{\beta}{\mathcal O'}$ and $\hat\pi_{W^\C}(\mathcal O)=\maxm{\beta}{\mathcal O}$.
Summing up, we have proved the following result.
\begin{prop}\label{flow-parabolic}
Let $W\subset \R^n$ be the eigenspace corresponding to the maximum eigenvalue of $\beta$. Then both $\hat\pi_{W} (\mathcal O )$ and $\hat\pi_{W^\C}(\mathcal O)$  are the unique compact orbit of the $G^{\beta+}$-action on $\mathcal O$ and $\hat\pi_{W^\C} (\OO')$ is the unique compact orbit of the $(U^\C)^{\beta+}$-action on $\mathbb P(\C^n)$.
\end{prop}
\section{$\rho$-connected subsets and the gradient map}\label{Satake-sempre-Satake}
Let $\mathcal E=\mathrm{conv}(\mup(\mathbb P (\R^n)))$.  This highly symmetric $K$-invariant convex body has been extensively studied in \cite{biliotti-ghigi-heinzner-2,bgh-israel-p}, see also \cite{ks,orbitope}  amongst  many others. Let $\lia \subset \liep$ be a maximal Abelian subalgebra. Since $P=\pi_\lia (\mathcal E)=\mua(\mathbb P (\R^n))=\mua(\mathcal O)$,  applying Theorem \ref{convex-reduction} it follows that  $\mathcal E$ coincides with the convex hull of $\mup(\mathcal O)$ and so the convex hull of a $K$-orbit. This means that $\mathcal E$ is a polar orbitope.  By Theorem \ref{meo-israel}, the face structure of $\mathcal E$, up to $K$-equivalence, is completely determined by the face structure of $P$  and any face of $\mathcal E$ is exposed. By a recent result, $\mathcal E$ is also a spectrahedron \cite{ks}.

Let $F$ be a face of $\mathcal E$. By Lemma \ref{face-chain}, there exists a chain of faces
\[
F=F_0 \subsetneq F_1 \subsetneq \cdots F_k \subsetneq \mathcal E.
\]
Since any face is exposed, there exists $\beta_0,\beta_1,\ldots,\beta_k \in \liep$ such that
\[
F_i=\mathrm{max}_{\mathcal E} (\beta_i):=\{z\in \mathcal E:\, \langle z , \beta_i \rangle =\mathrm{max}_{y\in \mathcal E} \langle \cdot , \beta_i \rangle \}
\]
Then
\[
\mup^{-1}(F_i)=\maxm{\beta_i}{\mathbb P (\R^n)}=\mathbb P(W_i),
 \]
where $W_i$ is the eigenspace corresponding to the maximum eigenvalue of $\beta_i$. Moreover,
\[
\maxm{\beta_i}{\mathcal O}=\maxm{\beta_i}{\mathbb P(\R^n)}\cap \OO \subseteq \maxm{\beta_i}{\OO'},
\]
and
\[
\mup( \maxm{\beta_i}{\mathcal O})=\mathrm{ext}\, F_i.
\]
Since $\mathbb P(W_i)$ is completely determined by $ \maxm{\beta_i}{\mathcal O}$, applying Theorem \ref{parabolici} we get the following result.
\begin{prop}\label{chain}
Given a chain of faces $F=F_0 \subsetneq F_1 \subsetneq \cdots \subsetneq F_k \subsetneq \mathcal E$, there exist two chains of submanifolds
\[
\begin{array}{ccccccccc}
\maxm{\beta_0}{\mathcal O'} & \subsetneqq & \maxm{\beta_1}{\mathcal O'} & \subsetneqq & \cdots & \cdots & \subsetneqq &  \maxm{\beta_k}{\mathcal O'} & \subseteq \mathbb P(\C^n) \\
\cup                                &                      & \cup                                &                      &           &            &         & \cup & \cup \\
\maxm{\beta_0}{\mathcal O} & \subsetneqq & \maxm{\beta_1}{\mathcal O} & \subsetneqq & \cdots & \cdots & \subsetneqq &  \maxm{\beta_k}{\mathcal O} & \subseteq \mathbb P(\R^n)\\
\end{array}
\]
such that the vertical inclusions are Lagrangian and totally geodesic immersions.
\end{prop}
Let $\lia \subset \liep$ be a maximal Abelian subalgebra. Then
\[
\mathfrak z (\lia)=\mathfrak m \oplus \lia,
\]
where $\mathfrak m = \mathfrak z (\lia)\cap \mathfrak k$. If $\lia'\subset \mathfrak m$ is a maximal Abelian subalgebra of $\mathfrak m$, then $\lia'+ i \lia \subset \mathfrak u=\liek + i \liep$ is a maximal Abelian subalgebra of $\liu$ and so $(\lia'+ i \lia)^\C \subset \lieg^{\C}=\liu^{\C}$ is a Cartan subalgebra. Given $\lia, \lia'$,  and $\Pi\subset \Delta(\lieg,\lia)$ be a basis one can choose
a basis of $(\lia + i\lia')^*$ adapted to $\Pi$  and $(i\lia')^*$. Indeed, it is possible to define a set of simple roots
of $\hat{\Delta}$ such that the projection of a subset of the simple roots of $\hat{\Delta}$ onto $\mathfrak a^*$ is equal to $\Pi$  (see \cite[$p. 51-52$]{gjt},\cite[$p. 272-273$]{helgason}). In particular the Borel subalgebra of $\lieg^{\C}$ is contained in
$q_{\emptyset}^\C=(\mathfrak m + \lia + \mathfrak n)^{\C}$.

Let $\tilde \mu_\rho$ the highest weight of $\lieg^{\C}$ with respect to the partial ordering determined $\hat{\Delta}$. Let $x_o=[v_\rho]$, where $v_\rho$ is any highest weight vector.  It is well-known that
\[
\mu:\mathbb P(\C^n) \lra \lia'\oplus i \lia, \qquad \langle \mu(x_o) , \xi \rangle )=\tilde{\mu_\rho}(\xi),
\]
see \cite{biliotti-ghigi-American} that has opposite sign convection for $\mu$, and \cite{baston-eastwood}.
By Proposition \ref{closed-orbit}, one can choose $v_\rho \in \R^n$. Hence $G\cdot x_o=\mathcal O$, and so
\[
\langle \mua(x_o),\xi \rangle=(i\tilde{\mu_\rho})(\xi).
\]
$(i\tilde{\mu_\tau})_{\vert{\lia}}$ is the highest weight of $\lieg$ with respect the induced order on $\lia^*$.
From now on, we denote by $\mu_\rho=(i\tilde{\mu_\rho})_{\vert{\lia}}$ and also the vector in $\lia$, which represent $(i \tilde{\mu_\rho})_{\vert{\lia}}$ with respect to the
scalar product $\scalo$.  Since $\mu_\rho$ belongs to the positive Weyl chamber defined by the order chosen, we get the following result.
\begin{prop}\label{weight}
The image of the gradient map $\mup:\mathbb P(\R^n) \lra \lia$, is the convex hull of the Weyl group orbit of the highest weight $\mu_\rho$. Hence the weights of $\rho$ are contained in $\mathrm{conv}(\mathcal W \cdot \mu_\rho)$.
\end{prop}
\begin{proof}
Denote the weights of $\rho$ by $\mu_1,\ldots,\mu_p \in \lia^*$. It is easy to check that
\[
\mua(\mathbb P(\R^n))=\mathrm{conv}(v_1,\ldots,v_p),
\]
where $v_1,\ldots,v_p \in \lia$ satisfy
\[
\langle v_i,H\rangle =\mu_i (H),
\]
for any $H\in \lia$ an for $i=1,\ldots,n$. By Proposition \ref{image} we have
\[
\mua(\mathbb P(\R^n))=\mua(\mathcal O).
\]
Since $\mu_\rho$ lies in the positive Weyl chamber, it follows that $\mua(\OO)\cap \lia_{+}=\mua(x_o)$. Applying the Kostant's convexity Theorem \cite{kostant-convexity}, we have
\[
\mua(\mathbb P(\R^n))=\mua(\mathcal O)=\mathrm{conv}(\mathcal W \cdot \mu_\rho),
\]
concluding the proof.
\end{proof}
Let $W\subset \R^n$. We say that $W$ is $\rho$-admissible if $Q(W):=\{g\in G:\, gW=W\}$ is a parabolic subgroup of $G$ and the $Q(W)$-action on $W$ is irreducible. If $Q(W)=G^{\beta+}$ then $\mathbb P (W)$ is the linear span of the unique compact $Q(W)$-orbit  in $\OO$ given by $\maxm{\beta}{\mathcal O}$ and  $\mup( \maxm{\beta}{\mathcal O})=\mathrm{ext}\, F_\beta (\mathcal E)=F_W$.
\begin{lemma}\label{tau-faccie}
Let $W,W'\subseteq V$ be $\rho$-connected subspaces. Then $W\subseteq W'$ if and only if $F_W\subseteq F_W'$. Moreover $W=W'$ if and only if $\mathrm{relint}\, F_W = \mathrm{relint}\, F_{W'}$.
\end{lemma}
\begin{proof}
If $W	\subseteq W'$, then $\mathbb P(W) \cap \OO \subseteq \mathbb P(W') \cap \OO$, and so
\[
\mathrm{ext}\, F_W=\mup(\mathbb P(W) \cap \OO) \subseteq \mup(\mathbb P(W') \cap \OO) \subseteq \mathrm{ext}\, F_{W'}.
\]
This implies $F_W \subseteq F_{W'}$ since any face is the convex hull of its extremal points. Vice-versa, if $F_W \subseteq F_{W'}$, then
\[
\mathbb P(W)=\mup^{-1}(F_W) \subseteq \mup^{-1}(F_{W'})=\mathbb P(W').
\]
By Theorem \ref{schneider-facce}, $F_W=F_{W'}$ if and only if  $\mathrm{relint}\, F_W = \mathrm{relint}\, F_{W'}$, concluding the proof.
\end{proof}
Since $\mup(\OO)$ is a $K$-orbit,  it is a fundamental fact that the action of K on $\mup(\OO)$  extends to an action of G see \cite[Prop. 6]{heinz-stoezel}. The set of extreme points of $F_W$, that we denote by $\mathrm{ext}\, F_W$, is contained in $\mup(\OO)$. We define
\[
Q_{F_W}:=\{h\in G:\, h \mathrm{ext}\, F_W = \mathrm{ext}\, F_W\}.
\]
\begin{prop}\label{stabilizzatore-faccie}
$Q_{F_W}=Q(W)$. Moreover, if $\beta \in \mathrm{C}_{F_W}^{H_F}$, then $G^{\beta+}=Q(W)$. Hence $Q(W)$ only depends on the face $F_W$.
\end{prop}
\begin{proof}
Let $\beta\in \mathrm{C}_{F_W}^{H_F}$. Then $\mup^{-1}(F_\beta (\mathcal E))=\mathbb P(W)$ and $W(G^{\beta+})=W$. This implies  $G^{\beta+}\subseteq Q(W)$. Since $G^{\beta+}=Q_{F_W}$ \cite[Proposition 3.8, $p.598$]{biliotti-ghigi-heinzner-2}, it follows that $Q_{F_W}\subseteq Q(W)$.

Let $g\in Q(W)$. By Proposition \ref{decomposition-parabolic}, $g=kp$ for some $k\in K$ and some $p\in Q_{F_W}$. Then $gW=W$ implies $kW=W$ and so,  keeping in mind that $\maxm{\beta}{\OO}= P (W)\cap \mathcal O$, $k$ preserves $\maxm{\beta}{\OO}$.
By the $K$-equivariance of the $G$-gradient map $\mup$, we get $k\, \mathrm{ext}\, F_W=\mathrm{ext}\, F_W$ and so $k\in H_F\subset Q_{F_W}$. This proves $Q(W)\subseteq Q_{F_W}$, concluding the proof.
\end{proof}
Set $\mathcal H (\rho)=\left\{(W,Q(W)): W\, \mathrm{is\ } \rho-\mathrm{connected\ subspace\ of\ }\R^n \right\}$. The following Lemma is easy to prove.
\begin{lemma}\label{action-tauconnected}
Let $W$ be a $\rho$-connected subspace of $V$ and let $g\in G$. Then
\begin{enumerate}
\item $gW$ is a $\tau$-connected subspace;
\item there exists $k\in K$ such that $gW=kW$;
\item $Q(gW)=gQ(W)g^{-1}=kQ(W)k^{-1}$
\end{enumerate}
\end{lemma}
$K$ acts on $\mathcal H (\tau)$ as follows:
\[
k(W,Q(W)):=(kW,kQ(W)k^{-1}).
\]
As in \cite{bsg}, one can prove the following result.
\begin{prop}\label{tau-connected-face}
Let $\mathscr{F} ( \mathcal E)$ denote the set of the faces of $\mathcal E$. Then the map
\[
\mathscr{Z} : \mathcal H (\tau) \lra \mathscr{F} ( \mathcal E), \qquad (W,Q(W)) \mapsto F_W,
\]
is $K$-equivariant and bijective.
\end{prop}
In the sequel we always refer to \cite[Section I.1 Real Parabolic subgroups]{borel-ji-libro}.

\begin{defin}
  A subset $I\subset \simple$ is $\mu_\rho$-\enf{connected} if
  $I\cup\{\mu_\rho \}$ is connected, i.e., it is not the union of subsets orthogonal with respect to the Killing form.
\end{defin}
Connected components of $I$ are defined as usual.  The $I$ is $\mu_\rho$-connected subset if and only if any connected component of $I$ contains at least one element $\alpha$  which is not orthogonal to $\mu_\rho$.
\begin{defin}
    If $I\subset \simple$ is $\mu_\rho$-connected, denote by $I'$ the
    collection of all simple roots orthogonal to $\{\mu_\rho \}\cup I$.
    The set $J:=I\cup I'$ is called the $\mu_\rho$-\enf{saturation} of
    $I$.
\end{defin}
In \cite{biliotti-ghigi-heinzner-2}, see also \cite{ks}, the following theorem is proved.
\begin{teo}\label{parabolici-faccie}$\, $
\begin{enumerate}
\item Let $I \subset \Pi$ be a $\mu_\rho$-connected subset and let $J$ be its $\mu_\rho$-saturation. Then $Q_I\cdot \mu_\rho=Q_J \cdot \mu_\rho$ and $F:=\mathrm{conv}(Q_I\cdot \mu_\rho)$ is a face of $\mathcal E$. Moreover,
$F=F_\beta (\mathcal E)=\mathrm{conv}(K^{\beta} \cdot \mu_\rho)$, where $\beta \in \lia$ satisfies $Q_I=G^{\beta+}$;
\item  let $\beta'\in \lia$ be such that $Q_J=G^{\beta'+}$. Then $F=F_{\beta'} (\mathcal E)=\mathrm{conv}(K^{\beta'}\cdot \mu_\rho)$. Moreover $\{g\in G:\, g\ext{F}=\ext{F}\}=Q_J=G^{\beta'+}$;
\item Any face of  $\mathcal E$ is conjugate to one of the faces constructed
    in (a).
\end{enumerate}
\end{teo}
Let $Q$ be a parabolic subgroup. We denote by $\OO(Q)$ the unique closed orbit of $Q$ in $\OO$By Theorem \ref{parabolici}, $\mup(\OO(Q))=\mathrm{ext}\, F$ for some face $F$ of $\mathcal O$. Moreover, $\mup^{-1}(F)=W$ is the unique $\rho$-connected subspace such that $Q$ acts irreducible on $W$,  $F=F_W$  and $Q\subset Q(W)$.

Let $I\subset \Pi$ be a $\mu_\rho$-connected subset and let $F_I$ be the face of $\mathcal E$ such that $\mathrm{ext}\, F_I=Q_I \cdot \mu_\rho$. We also denote by $W_I=\mup^{-1}(F_I)$ be the unique subspace of $\R^n$ such that $Q_I$ acts irreducible.
\begin{prop}\label{parabolic-coniugation}
If $Q$ is a parabolic subgroup of $G$, then  there exists $\mu_\rho$-connected subset $I\subset \Pi$ and $k\in K$ such that $W_I=kW$ and $
Q_I\subseteq kQk^{-1}\subseteq Q_J.$
\end{prop}
\begin{proof}
Let $W$ be the unique subspace of $V$ be such that $Q$ acts irreducibly on $W$. By Theorem \ref{parabolici-faccie}, there exists $k_1\in K$ such that $W=k_1W_I$ and $Q\subset k_1 Q_J k_1^{-1}$. Let $\OO(W)$ denote the unique compact $Q$-orbit contained in $\OO$. Since $k_1\OO (Q)=\OO(Q_I)$, keeping in mind that $\mup(\OO(Q))=\mathrm{ext}\,F_W$ and $\mup(\OO (Q_I))=\mathrm{ext}\, F_I$ and $\mup$ is $K$-equivariant, it follows
\[
\mathrm{ext}\, F_I=k_1^{-1} \mathrm{ext}\, F_W .
\]
By \cite[Proposition 3.5]{biliotti-ghigi-heinzner-2}, there exists $k_2 \in K$ such that  $\lia \subset k_2 \mathfrak q k_2^{-1}$ and  $(k_2 \mathrm{ext}\, F_W) \cap \lia$ is a face of $P$. By the main result  proved in \cite{biliotti-ghigi-heinzner-2} $\mathscr F (P)/\mathcal W \cong \mathscr F (\mathcal E) /K$. Hence, there exists $\theta \in N_K (\lia)$ such that
\[
\mathrm{ext}\, F_I \cap \lia=(\theta k_2 \mathrm{ext} F_W)\cap \lia.
\]
By \cite[Theorem 1.1]{biliotti-ghigi-heinzner-2}, we have $\mathrm{ext}\, F_I=\theta k_2 \mathrm{ext} F_W$.

Set $k=\theta k_2$. Since $\lia \subset k\lieq k^{-1}$, there exists a $\tilde J$ subset of $\Pi$ such that $Q_{\tilde J}=kQk^{-1}$. By Theorem \ref{parabolici-faccie} and Proposition \ref{gi}, we have
\[
Q_{\tilde J} \cdot \mu_\rho=Q_I \cdot \mu_\rho=\ext F_I.
\]
Let $\tilde I$ denote the maximal $\{\mu_\rho\}$-connected subset contained in $\tilde J$. By Theorem \ref{parabolici-faccie}, $Q_{\tilde I} \cdot \mu_\rho$ is a face and
$Q_{E}\cdot \mu_\rho=Q_{\tilde I}\cdot \mu_\rho$, where $E$ is the saturation of $\tilde I$. Since $\tilde I \subseteq \tilde J \subseteq E$, it follows
\[
Q_{\tilde I}\cdot \mu_\rho =Q_{\tilde J} \cdot \mu_\rho=Q_{E} \cdot \mu_\rho.
\]
On the other hand $Q_I\cdot \mu_\rho= Q_{\tilde J} \cdot \mu_\rho$ and so, by Theorem \ref{parabolici-faccie}, $Q_E =Q_{J}$.
This implies $E=J$, $\tilde I=I$  due to the fact that $I$, respectively $\tilde I$, is the maximal $\{\mu_\rho\}$-connected subspace of $J$, and
\[
Q_I\subset Q_{\tilde J}\subset Q_J,
\]
where $I\subseteq \tilde J \subseteq J$.
\end{proof}
Let $I'=\tilde J \setminus\{I\}$. The set $I'$ is perpendicular to $I$ and so the Langlands decomposition of $Q_{\tilde J}$ can be written as
\[
Q_{\tilde J}=N_{I} A_I M_I M_{I'},
\]
see \cite{borel-ji-libro}. By \cite[Lemma I.4.25, p. $69$]{borel-ji-libro}, $M_{I'}$ acts trivially on $W_I$. Hence, the $Q_{\tilde J}$-action on $W_I$ is completely determined by the $Q_I$-action on $W_I$. Hence, keeping in mind Corollary \ref{oi}, we have proved the following result.
\begin{teo}
The face structure of $\mathcal E$, up to $K$-equivalence, describes the irreducible representations of parabolic subgroups of $G$  induced by $\rho:G \lra \mathrm{SL}(n,\R)$ and so the irreducible representations of parabolic subgroups of $G$ induced by $\rho:G \lra \mathrm{SL}(n,\C)$.
\end{teo}
Let $I\subset \Pi$ be a $\mu_\rho$-connected subset and let $F_I$  the face of $\mathcal E$ be such that $\mathrm{ext}\, F_I=Q_I \cdot \mu_\rho$. Let $\mathfrak a_I :=\bigcap_{\alpha \in I} \, \mathrm{Ker}\, \alpha$ and let $\mathfrak a^I$  be the orthogonal complement of $\mathfrak a_I$ in $\mathfrak a$.  Then
$
\mathfrak q_I=\mathfrak n_I \oplus \mathfrak a_I \oplus \mathfrak m_I,
$
where $\mathfrak m_I=\mathfrak z_{\mathfrak k}(\mathfrak a)\oplus \mathfrak a^I \bigoplus_{\alpha \in I} \mathfrak g_\alpha$ is the Lie algebra of a Levi factor of $Q_I$, that we denote by $M_I$ which is not connected in general. We recall that $\mathfrak g_{\alpha}:=\{v\in \lieg:\, [H,v]=\alpha(H)v,\, \forall H\in \lia\}$. $M_I$ is compatible and $K_I=K\cap Q_I$  is its maximal compact subgroup. The Abelian subalgebra $\mathfrak a^I$ is a maximal Abelian subalgebra of $\mathfrak m_I\cap \liep$. Let  $\mathcal W_I=N_{K_I}(\lia^I)$. $\mathcal W_I$  is the subgroup of the Weyl group generated by the the root reflections induced by the elements of $I$. We split $\mu_\rho=y_0+y_1$, where $y_0 \in \mathfrak a_I$ and $y_1\in \mathfrak a^I$.  By Proposition 3.5 in \cite{biliotti-ghigi-heinzner-2}, we have
\[
Q_I \cdot \mu_\rho=K_I\cdot \mu_\rho =y_0+K_I \cdot y_1.
\]
By the Kostant convexity Theorem \cite{kostant-convexity},
\[
\pi_\lia(Q_I \cdot \mu_\rho)=y_0+\mathrm{conv}(\mathcal W_I \cdot y_1).
\]
By Proposition \ref{gi}, $Q_I \cdot x_o$ is the unique compact $Q_I$-orbit on $\OO$ and $\mup(Q_I \cdot x_o )=Q_I\cdot \mu_\rho$. Therefore
\[
\mua(Q_I \cdot x_o)=y_0+\mathrm{conv}(\mathcal W_I \cdot y_1).
\]
Summing up, keeping in mind Theorem \ref{parabolici-faccie}, we  get a result proved by Casselman \cite{casselman} using the $G$-gradient map.
\begin{teo}\label{casselman}
The map $I \mapsto \mathrm{conv}\, (\mathcal W_I \cdot \mu_\rho)$ induces a bijection between the $\mu_\rho$-connected subset of $\Pi$ and the faces of $P$ up to the Weyl-group action. Moreover,
\[
\mua(Q_I \cdot x_o)=\mathrm{conv}\, (\mathcal W_I \cdot \mu_\rho)=  y_0+\mathrm{conv}\, (\mathcal W_I \cdot y_1 ),
\]
where $y_0 \in \mathfrak a_I$, $y_1\in \mathfrak a^I$ and $\mu_\rho=y_0+y_1$.
\end{teo}
\section{Norm square gradient map}
Let $\mup:\mathbb P (\R^n) \lra \liep$ be the $G$-gradient map and let
$\nu_\liep:=\frac{1}{2}\parallel \mup (x) \parallel^2$ denote the norm square gradient map. Let $\nu_\liu (p):=\frac{1}{2}\parallel \mu (p) \parallel^2$ denote the norm square momentum map. The gradient of $\nu_\liep$, respectively $\nu_\liu$, is given by $\mathrm{grad}\, \nu_\liep (x)= \mu(x)_{\mathbb P (\R^n)}$, respectively $\mathrm{grad}\, \nu_\liu (y)= \mu(x)_{\mathbb P (\C^n)}$ \cite{heinzner-schwarz-stoetzel}. Lemma \ref{restrizione} implies  $x\in \mathbb P(\R^n)$ is a critical point of $\nu_\liep$ if and only if $x\in \mathbb P(\R^n)$ is a critical point of $\nu_\liu$, see also \cite{jabo}. Proposition \ref{closed-orbit} implies
\[
\mathrm{Max}_{x\in \mathbb P(\R^n)} \nu_\liep =\mathrm{Max}_{x\in \mathbb P(\C^n)} \nu_\liu,
\]
and so the maximum of $\nu_\liu$ is achieved on $\mathbb P(\R^n)$. We shall investigate the strata associated to the critical $K$-orbits of $\nu_\liep$.
The negative gradient flow line of $\nu_\liep$ through $x_0\in \mathbb P(\R^n)$ is the solution of the differential equation
\[ \left\{ \begin{array}{ll}
         \dot{x}(t) = -\beta_{\mathbb P (\R^n)} (x(t)), \quad t\in \mathbb{R} \\
        x(0) = x_0.\end{array} \right.
\]
The Lojasiewicz gradient inequality holds and so the flow is defined for any $t\in \R$ and the limit
$$x_\infty := \lim_{t \rightarrow +\infty} x(t) $$
exist \cite[Theorem $3.3$]{bjmain}, see also \cite{grs} for the $U^\C$-action on $\mathbb P(\C^n)$  and the corresponding momentum map. Moreover,
$$
\parallel \mu_\mathfrak{p}(x_\infty)\parallel = \text{Inf}_{g\in G}\parallel \mu_\mathfrak{p}(gx_0)\parallel
$$
and the $K$-orbit of $x_\infty$ depends only on the $G$-orbit of $x_0$ \cite[Theorem $4.7$]{bjmain}, see also \cite{jabo}.

By Lemma \ref{restrizione}, $\mup=i\mu$ on $\mathbb P(\R^n)$. Then the negative gradient flow line of $\nu_\liep$ through $x_0$ coincides with the negative gradient flow of $\nu_\liu$.  Hence, given
$x_0 \in \mathbb P(\R^n)$, we get
\[
\begin{split}
\parallel \mu_\mathfrak{p}(x_\infty)\parallel &=\text{Inf}_{g\in G}\parallel \mu_\mathfrak{p}(gx_0)\parallel \\ &=\text{Inf}_{g\in G^\C} \parallel \mu_(gx_0)\parallel.
\end{split}
\]
Let $x\in \mathbb P(\C^n)$.  Applying the negative gradient flow of the norm square momentum map, the orbit $U^\C \cdot x$ flows on the $U$-orbit throughout $x_\infty=\lim_{t\mapsto +\infty} x(t)$.
If  $U^\C \cdot x \cap \mathbb P(\R^n)\neq \emptyset$ then $U\cdot x_\infty \cap \mathbb P(\R^n) \neq \emptyset$. Let $y\in U\cdot x_\infty \cap \mathbb P(\R^n)$. Then, keeping in mind that the norm square momentum map is $U$-invariant, we have
\[
\parallel \mu(x_\infty)=\parallel \mu(y) \parallel=\parallel \mup(y) \parallel,
\]
and so
\[
\mathrm{Inf}_{g\in U^\C}\parallel \mu (gx)\parallel =\mathrm{Inf}_{y\in U^\C \cdot x \cap \mathbb P(\R^n)} \parallel \mup (y)\parallel.
\]
Summing up, we have proved the following result.
\begin{prop}
Let  $x_0 \in \mathbb P(\R^n)$. Then
\[
\begin{split}
\parallel \mu_\mathfrak{p}(x_\infty)\parallel &=\text{Inf}_{g\in G}\parallel \mu_\mathfrak{p}(gx_0)\parallel \\ &=\text{Inf}_{g\in G^\C} \parallel \mu_(gx_0)\parallel.
\end{split}
\]
Let $x\in \mathbb P(\C^n)$ be such that  $U^\C \cdot x \cap \mathbb P(\R^n) \neq \emptyset $. Then
\[
\mathrm{Inf}_{g\in U^\C}\parallel \mu (gx)\parallel =\mathrm{Inf}_{y\in U^\C \cdot x \cap \mathbb P(\R^n)} \parallel \mup (y)\parallel.
\]
\end{prop}
Let $K\cdot \beta$ be  a critical orbit of the norm square gradient map.  By Lemma \ref{restrizione}, $U\cdot i \beta$ is a critical orbit of the norm square momentum map. Moreover, by the above discussion,  the stratum associated to  $K\cdot \beta$  is contained in the stratum associated to  $U\cdot i \beta$.  We recall that the stratum associated to  $U\cdot i \beta$ is a smooth analytic subset of $\mathbb P(\C^n)$ \cite{kirwan}.
Assume that the stratum associated to $K\cdot \beta$ is open.  Since $\mathbb P(\R^n)$ is a Lagrangian submanifold of $\mathbb P(\C^n)$, it follows that the stratum associated to $U\cdot i\beta $ is open as well. It is well-known that there exists a unique open stratum of $\nu_\liu$ in $\mathbb P(\C^n)$. It is the minimal stratum \cite{kirwan}. Therefore there exists a unique open stratum for $\nu_\liep$ which is the minimal stratum. In particular,
\[
\mathrm{Inf}_{x\in \mathbb P(\R^n)} \nu_\liep = \mathrm{Inf}_{x\in \mathbb P(\C^n)} \nu_{\mathfrak u}.
\]
Since the other strata of the norm square gradient map has codimension at least one, the minimal strata is dense. Summing up, we have proved the following result.
\begin{prop}
The norm square gradient map $\nu_\liep$ has a unique open stratum, which is the minimal stratum. This stratum is open, dense and is given by the intersection of $\mathbb P(\R^n)$ and the minimal stratum of  $\nu_\liu$.
\end{prop}

\end{document}